\documentclass[reqno, 11pt]{amsart}
\usepackage[numbers, square]{natbib}
\usepackage{amssymb}
\usepackage{amsmath}
\usepackage{amsfonts}
\usepackage{graphicx}
\usepackage{amsthm}
\usepackage{hyperref}
\usepackage[dvipsnames]{xcolor}
\usepackage{esint}

\usepackage{color}
\hypersetup{
  colorlinks,
  citecolor=blue,
  linkcolor=red,
  urlcolor=blue}
  
\usepackage[left=1.4in, right=1.4in, bottom=1.4in]{geometry}
\newtheorem{theorem}{Theorem}[section]
\newtheorem{lemma}[theorem]{Lemma}
\newtheorem{proposition}[theorem]{Proposition}
\newtheorem{corollary}[theorem]{Corollary}
\theoremstyle{definition}

\newtheorem{remark}[theorem]{Remark}

\numberwithin{equation}{section}

\usepackage{dsfont}

\newcommand{\OO}{\mathcal{O}}

\newcommand{\vertiii}[1]{{\left\vert\kern-0.25ex\left\vert\kern-0.25ex\left\vert #1 
    \right\vert\kern-0.25ex\right\vert\kern-0.25ex\right\vert}}

\begin{document}
\title[Observability inequalities for heat equations with potentials]
{The observability inequalities for heat equations\\ with  potentials }
\author{ Jiuyi Zhu}
\address{Jiuyi Zhu:
Department of Mathematics\\
Louisiana State University\\
Baton Rouge, LA 70803, USA\\
Email:  zhu@math.lsu.edu }
\author {Jinping Zhuge}
\address{Jinping Zhuge:
 Morningside Center of Mathematics\\
the Academy of Mathematics and Systems Science\\ Chinese Academy of Sciences\\
Beijing, China\\
Email: jpzhuge@amss.ac.cn }
%\date{\today}
%\dedicatory{}

\maketitle

\begin{abstract}
This paper is mainly concerned with the observability inequalities for heat equations with time-dependent Lipschitz potentials. The observability inequality for heat equations asserts that the  total energy of a solution is bounded above by the energy localized in a subdomain with an observability constant. For a bounded measurable potential $V = V(x,t)$, the  factor in the observability constant arising from the Carleman estimate is best known to be $\exp(C\|V\|_{\infty}^{2/3})$ (even for time-independent potentials). In this paper, we show that, for Lipschtiz potentials, this factor can be replaced by $\exp(C(\|\nabla V\|_{\infty}^{1/2} +\|\partial_tV\|_{\infty}^{1/3} ))$, which improves the previous bound $\exp(C\|V\|_{\infty}^{2/3})$ in some typical scenarios. As a consequence, with such a Lipschitz potential, we obtain a quantitative regular control in a null controllability problem.
In addition, for the one-dimensional heat equation with some time-independent bounded measurable potential $V = V(x)$, we obtain the observability inequality with optimal constant on arbitrary measurable subsets of positive measure both in space and time.

%\noindent\textbf{Keywords:} Observability inequality, null controllability, Carleman estimates.
\end{abstract}

%\jz{Major revision suggestions: (1) As our previous paper, the condition on the potential can be relaxed to $V = V_1 + V_2$ with $V_1$ being differentiable. The added proof is just a few lines}.

%(2) Proposition \ref{prop.HUM} seems standard. We can just provide a sketch of the proof with suitable references to make the paper more concise. \textcolor{blue}{Let us keep it or you can say ``Readers who are familiar with the technique can skip it'', or we move it to the appendix. I wrote it to make sure I knew the proof} \jz{OK! Let's keep it in appendix}

%(5) Add explanation of the relationships between the regular control in this paper and the ``bang-bang'' control in literature.
% \textcolor{blue} {Could you describe some ``bang-bang'' control. I only know a little.} \jz{I will find some references later.}

%\jz{Overall, the paper focuses on the proofs of Lemma \ref{Carle-main}, Theorem \ref{th2} and  Theorem \ref{th3}. We must clarify what is new and what is old so that the novelties are highlighted and the readers are not distracted by old or classical results.}

\section{Introduction}
In this paper, we mainly study the observability inequality for the heat equation with a time-dependent potential $V(x,t)$,
\begin{equation}
    \left\{
    \begin{aligned}
        y_t-\Delta y+V(x,t)y &= 0 \quad  &\text{in }& \ Q_T:=\Omega\times (0, T) , \\
        y &=0  \quad &\text{on }& \ \Sigma_T:= \partial\Omega\times (0, T), \\
        y(\cdot,0) &= y_0    &\text{on }& \  \Omega,
    \end{aligned}
    \right.
    \label{nonlinear-core}
\end{equation}
 where $\Omega\subset \mathbb R^n$ is an open bounded smooth domain and $n$ is the dimension. We denote by $\|\cdot\|_\infty$ or $\|\cdot\|_{L^\infty}$ the norm of $L^\infty(Q_T)$.{Throughout this paper, we will consider the potential $V$ with a large norm (say, $\|V\|_\infty\geq 10$).} Let $\omega\subset \mathbb R^n$ be a nonempty open set in $\Omega$. For $V(x,t)\in L^\infty(Q_T)$, it was shown in \cite{FZ00a} that the solution $y$ of (\ref{nonlinear-core}) satisfies the observability inequality
\begin{align}
\|y(x,T)\|_{L^2(\Omega)}\leq \exp\big( C (1+T^{-1}+T\|V\|_\infty+ \|V\|^{2/3}_\infty)\big) \|y\|_{L^2(\omega\times (0, T))},
\label{observe-key}
\end{align}
 where $C=C(\Omega, \omega)$ depends on $\Omega$ and $\omega$ but independent of $T$ and $V$.
 The observability inequality (\ref{observe-key}) plays an important role in the control problem for linear and nonlinear equations; see e.g., \cite{FZ00a,FZ00b}. 
 
Let us examine the components in the observability constant in (\ref{observe-key}),
 \begin{align*}
 %\exp\big( C (1+T^{-1}+T\|V\|_\infty+ \|V\|^{2/3}_\infty)\big)=
 \exp \big(C(1+ T^{-1} )\big) \exp(CT\|V\|_\infty) \exp(C\|V\|^{2/3}_\infty).
 \end{align*}
The first two factors of the above  constant seem necessary in the observability inequalities.
In the absence of the potential, i.e., $V(x,t)=0$, the observability constant is simply $\exp \big(C(1+T^{-1})\big)$. This growth for small $T$ seems to be optimal, which roots in the heat kernel; see e.g., \cite{M04}. For nontrivial potentials, the factor $\exp(CT\|V\|_\infty)$ in the observablity constant comes from the dissipativity of energy in time for heat equations; see e.g., Lemma \ref{lemm-11}. The time-independent factor $ \exp(C\|V\|^{2/3}_\infty)$ appears to be the most interesting one, arising from the Carleman estimates. We will focus on the study of this factor in the observability inequality (\ref{observe-key}).  For $\|V\|_\infty^{-2/3}\lesssim T \lesssim \|V\|_\infty^{-1/3}$, the factor $ \exp(C\|V\|^{2/3}_\infty)$ in the observability inequality (\ref{observe-key}) plays the predominant role. It was shown in \cite{DZZ08} that the growth rate  $ \exp(C\|V\|^{2/3}_\infty)$ is sharp within $T\lesssim \|V\|_\infty^{-1/3}$ for parabolic systems in even spacial dimensions for complex-valued potentials $V(x,t)$. This is based on Meshkov's  counter-example in \cite{M92}, which violates the original Landis' conjecture \cite{KL88} on the optimal decay at infinity for the solution of elliptic equation,
 \begin{align}
 \Delta u+a(x) u=0 \quad \mbox{in} \ \mathbb R^n,
 \label{mesh}
 \end{align}
 with a bounded complex-valued potential $a(x)$. In particular, Meshkov showed that there exists a nontrivial complex-valued solution of (\ref{mesh}) in $\mathbb R^2$ satisfying $|u(x)|\leq \exp(-c |x|^{4/3}).$ On the contrary, Meshkov  also showed that if a solution of (\ref{mesh}) satisfies $|u(x)| \le C_\tau \exp(-\tau |x|^{4/3})$
 for all $\tau>0$, then $u\equiv 0$. Nevertheless, the general Landis' conjecture remains open for bounded real-valued potentials, which asserts a solution of \eqref{mesh} must be trivial if it decays faster than $\exp(-|x|^{1+\varepsilon})$ for some $\varepsilon>0$.
Recently, Landis' conjecture was settled in \cite{LMNN20} for $n=2$. See also the progress for the conjecture for nonnegative potential $V(x)$ in \cite{KSW15}.

Analogous to the Landis' conjecture for (\ref{mesh}), it is an interesting open problem as asked in \cite{DZZ08} that if one can replace  $\exp(C\|V\|^{2/3}_\infty)$ by the optimal constant $\exp(C\|V\|^{1/2}_\infty)$ in (\ref{observe-key}) for the scalar heat equations. This  optimal constant is also crucial for a long-standing conjecture on the null controllability of nonlinear heat equations; see e.g., \cite{DZZ08} and references therein.
In this paper, we aim to improve the factor $\exp(C\|V\|^{2/3}_\infty)$ in the observability constant in some scenarios for Lipschitz potentials. Our main result is the following theorem.
\begin{theorem}\label{th1}
Let $V(x,t)$ be Lipschitz continuous in $Q_T$ and $y$ the solution of \eqref{nonlinear-core}. Then there exists a constant $C=C(\Omega, \omega)$ such that
\begin{align}
\| y(\cdot, T)\|_{L^2(\Omega)}\leq  \exp\{C(1+T^{-1}+T \|V\|_{\infty}+\| \nabla V\|^{1/2}_{\infty}+\|\partial_t V\|^{1/3}_{\infty})\}
\|y\|_{L^2(\omega\times (0, T))}.
\label{obser-lip}
\end{align}
\end{theorem}

Actually the observability  constant in the inequality (\ref{obser-lip}) is shown to be
\begin{equation*}
    \exp\{C(1+T^{-1}+T \|V\|_{\infty}+\|V\|^{1/2}_{\infty}+\| \nabla V\|^{1/2}_{\infty}+\|\partial_t V\|^{1/3}_{\infty})\}.
\end{equation*}
We have skipped the term $\|V\|^{1/2}_{\infty}$ in (\ref{obser-lip}) due to the fact $2 \| V\|_\infty^{1/2} \le T^{-1} + T\| V\|_\infty$.
It is easy to see the observability constant (\ref{obser-lip}) is better than the one in (\ref{observe-key}) in some cases.
Typically, if $V(x,t)\equiv \lambda$ is some large constant, the observability inequality (\ref{obser-lip}) is obviously sharper than (\ref{observe-key}), since $|\nabla V|=\partial_t V=0$. If we consider a time-dependent potential  $V(x,t)= V_0(x) (1+t)^\beta$ for some $\beta > 0$ and some Lipschitz potential $V_0$ with $\| V\|_\infty \simeq \| \nabla V\|_\infty$, then $\| \nabla V\|^{1/2}_{\infty}+\|\partial_t V\|^{1/3}_{\infty}$ is much smaller than $\| V\|^{2/3}_\infty $ for large $T$. We will apply Carleman estimates for the proof of Theorem \ref{th1}. Different from the arguments in \cite{FZ00a}, \cite{FI96}, we will obtain quantitative Carleman estimates which take the Lipschitz norm of $V(x,t)$ into account and aim to find a sharp lower bound of the Carleman parameter $\tau$. See Lemma \ref{Carle-main} for the details.

\begin{remark}
The observability inequality (\ref{obser-lip}) shares the similar spirit with the doubling inequality for elliptic equation $\Delta u+a(x) u=0$ on smooth compact manifolds $\mathcal{M}$.
%, as the sharp counter-examples constructed in \cite{DZZ08} rely on the decay at infinity for elliptic equations studied in \cite{M92}. 
For Lipschitz potentials,
it has been shown that the doubling inequality holds
\begin{align}\label{doubling-1}
\|u\|_{L^2(\mathbb
B_{2r}(x))}\leq  \exp(C(1+\|a\|^{1/2}_{\infty} + \| \nabla a\|_\infty^{1/2}) ) \|u\|_{L^2(\mathbb
B_{r}(x))}
\end{align}
for all geodesic balls $\mathbb B_{2r}(x)\subset \mathcal{M}$; see e.g., \cite{DF88,Zh16}.
For bounded measurable potentials, the best known doubling inequality for $n\geq 3$ \cite{BK05} is
\begin{align}\label{doubling-2}
\|u\|_{L^2(\mathbb
B_{2r}(x))}\leq  \exp(C(1+\|a\|^{2/3}_{\infty})) \|u\|_{L^2(\mathbb
B_{r}(x))}.
\end{align}
Except for $n=2$ (see \cite{KSW15,LMNN20}), it is open that if the exponent $2/3$ can be improved to $1/2$ in \eqref{doubling-2}. Note that \eqref{doubling-2} and \eqref{doubling-1} can be compared to \eqref{observe-key} and \eqref{obser-lip}, respectively.
We also mention that for the optimal constants in the three-cylinder inequalities of parabolic equations, the regularity of potential functions has been seen to play a role; see e.g., \cite{Zh18,GK18}.
\end{remark}

\begin{remark}
    The Lipschitz assumption on $V$ can be relaxed to bounded measurable potentials by using the idea of \cite{ZZ23}. Precisely, suppose that we can decompose $V = V_1 + V_2$, where $V_1$ is a Lipschitz component and $V_2$ is a bounded measurable component. Then with suitable modifications to our proof, we can obtain an observability constant as
    \begin{equation}
        \exp\{C(1+T^{-1}+T \|V\|_{\infty}+\| \nabla V_1\|^{1/2}_{\infty}+\|\partial_t V_1\|^{1/3}_{\infty} + \| V_2 \|_\infty^{2/3}) \}.
    \end{equation}
    This constant could improve the one in \eqref{observe-key} in some cases of bounded measurable potentials, such as a regular principle part $V_1$ perturbed by a rough small $V_2$.
\end{remark}

 The observability inequality as (\ref{observe-key}) plays a key role in the null controllability problem,
\begin{equation}
    \left\{
    \begin{aligned}
        y_t-\Delta y+V(x,t)y &= h \mathds{1}_\omega \quad  &\text{in }& Q_T, \\
        y &=0  \quad &\text{on }& \ \Sigma_T, \\
        y(\cdot,0) &= y_0    &\text{on }& \  \Omega.
    \end{aligned}
    \right.
    \label{linear-core-1}
\end{equation}
Here $\mathds{1}_\omega$ is the characteristic function of $\omega$, and $T$ is any given positive time.  Assume that the initial state $y_0\in L^2(\Omega)$. We aim to find a control $h\in L^2(\omega\times (0, T))$ such that the associated state $y$ possesses a desired behavior at time $t=T$. We say (\ref{linear-core-1}) is null controllable at a given time $T$ or the null controllability holds for (\ref{linear-core-1}), if the associated state $y$ satisfies
\begin{align}
y(\cdot, T)=0\quad \mbox{in} \ \Omega.
\end{align}

The null controllability for linear heat equations has been well studied; see e.g., the surveys \cite{FI96,Zua07,FG06}. Since we have obtained a new observability inequality for the heat equations with a Lipschitz potential, it is natural to ask for a more regular control function $h$ for the null controllability problem \eqref{linear-core-1} such that the trajectory $y(x,t)$ is a classical solution and the control cost is bounded by the observablity constant in (\ref{obser-lip}). Precisely, we have the following theorem.

%we aim to find if a regular control function exists such that the null controllability holds and the trajectory $y(x, t)$ is a classical solution, and if the control cost is bounded by the constant as the observability inequality (\ref{obser-lip}). We assume $y_0\in C^{2+\alpha}(\overline\Omega)$ for $0<\alpha<1$. Note that such a strong  regularity assumption is not necessary for null controllability. Please see Remark 1. We need it for the sake of the control cost estimates (\ref{cost}). We are able to show the following theorem. We hope that it will be useful in some other problems.
%\jz{To make the theorem clean, let us still assume $y_0 \in L^2$. In this setting, the $C^{2+\alpha,1+\alpha/2}$ estimate of $y$ blows up as $t\to 0$. }
\begin{theorem}\label{th2}
Let $\Omega$ be a bounded domain of class $C^3$.
For every $T>0$, there exists $h\in C^{\alpha ,\frac{\alpha}{2}}(\omega \times (0,T))$ with any $\alpha \in (0,1)$ such that
the equation \eqref{linear-core-1} is null controllable at time $T$. Furthermore,
\begin{align}
\|h\|_{C^{\alpha, \alpha/2}(\omega \times (0,T))}
\leq  \exp\{C(1 + T^{-1}+ T \|V\|_{\infty} + \| \nabla V\|^{1/2}_{\infty} + \|\partial_t V\|^{1/3}_{\infty}) \} \|y_0\|_{L^2(\Omega)},
\label{cost}
\end{align}
where $C$ depends only on $\Omega,\omega$ and $\alpha$.
\end{theorem}

Next, we consider the particular case $n=1$ with a time-independent bounded measurable potential $V = V(x)$.
For the one-dimensional wave equation with a bounded measurable potential, the optimal factor $\exp{(C\|V\|_\infty^{1/2})}$ has been obtained in \cite{Zua93}, using the sidewise energy estimate of wave equations. A note (with no proof) in \cite{DZZ08} indicates that the optimal constant $\exp{(C\|V\|_\infty^{1/2})}$ can be reached as well for the one-dimensional heat equations, using some well-known transformation from wave into heat processes.
Inspired by the recent progress on the Landis' conjecture for (\ref{mesh}) and the deep interplay between (\ref{nonlinear-core}) and (\ref{mesh}), we will establish the optimal observability inequality  for the one-dimensional heat equation, without referring to the wave equations. Furthermore, the observability inequalities will be obtained on arbitrary measurable sets of positive measure both in space and time, which seems not possible in the previous literature.
Precisely, consider the following heat equation
\begin{equation}
    \left\{
    \begin{aligned}
        y_t-\partial_{xx} y+V(x)y &= 0 \quad  &\text{in }& \ \Omega\times (0, T) , \\
        y &=0  \quad &\text{on }& \ \partial\Omega\times (0, T), \\
        y(\cdot,0) &= y_0    &\text{on }& \  \Omega.
    \end{aligned}
    \right.
    \label{nonlinear-one}
\end{equation}
where $\Omega=(0, \frac12)$ and $V\in L^\infty(\Omega)$ is  a real-valued function. We will show a spectral inequality for (\ref{nonlinear-one}) over measurable sets of positive measure in space and then adopt the strategy in \cite{PW13} for measurable sets of positive measure in time. The observation region therefore is restricted over the product of a subset of positive measure in time and a subset of positive measure in space. 
%We will distinguish between nonnegative potentials and general potentials that changes signs. For the nonnegative potentials, 
We will adapt the ideas in \cite{KSW15} for the answer of Landis' conjecture with nonnegative potentials and reduce the general sign-changing potentials to nonnegative potentials for heat equations.
%For general potentials, we will use the estimates in \cite{LMNN20} for the proof of Landis' conjecture.
{In the following theorem, we assume $\| V\|_\infty \geq 10$ and an obvious modification is needed if $\| V\|_\infty < 10$.}

\begin{theorem}
Let $E\subset (0, T)$ be a measurable set of positive measure and $\omega$ be a measurable subset of positive measure in $\Omega$. 
% Case 1: If $V(x)\geq 0$, then any solution $y(x,t)$ of \eqref{nonlinear-one} satisfies
% \begin{align}
% \| y(\cdot, T)\|_{L^2(\Omega)}\leq \exp( C(E, \Omega, \omega)+C(\Omega, \omega) {\|V\|^{1/2}_\infty})
% \|y\|_{L^2(\omega\times E)},
% \end{align}
% where $ C(E, \Omega, \omega)$ depends on $E$, $\Omega$ and $\omega$, and $ C(\Omega, \omega)$ depends on $\Omega$ and $\omega$.
If $V\in L^\infty(\Omega)$, then any solution $y(x,t)$ of \eqref{nonlinear-one} satisfies
\begin{align}
\| y(\cdot, T)\|_{L^2(\Omega)}\leq \exp( C(E, \Omega, \omega)+T \|V_-\|_{\infty}+ C(\Omega, \omega) {\|V\|^{1/2}_\infty})
\|y\|_{L^2(\omega\times E)},
\end{align}
where $V_- = \min\{V, 0 \}$ (the negative part of $V$), $ C(E, \Omega, \omega)$ depends on $E$, $\Omega$ and $\omega$, and $ C(\Omega, \omega)$ depends on $\Omega$ and $\omega$.

\label{th3}
\end{theorem}

In particular, if $E=(0, T)$, we have the following corollary.

\begin{corollary}
Let $\omega$ be a measurable subset of positive measure in $\Omega$ and $V\in L^\infty(\Omega)$. 
% Case 1: If $V(x)\geq 0$,
% then any solution $y(x,t)$ of \eqref{nonlinear-one} satisfies
% \begin{align}\label{obs-0T-V>0}
% \| y(\cdot, T)\|_{L^2(\Omega)}\leq  \exp( C(T^{-1}+\|V\|^{1/2}_{\infty}))
% \|y\|_{L^2(\omega\times (0, T))},
% \end{align}
% where $ C$ depends on $\Omega$ and $\omega$.
Then any solution $y(x,t)$ of \eqref{nonlinear-one} satisfies
\begin{align}\label{obs-0T-V}
& \| y(\cdot, T)\|_{L^2(\Omega)}
\leq  \exp( C(T^{-1} +T\|V_-\|_{\infty}+ {\|V\|^{1/2}_\infty} ))\|y\|_{L^2(\omega\times (0, T))}, 
\end{align}
where $ C$ depends on $\Omega$ and $\omega$.
\label{cor1}
\end{corollary}

 The organization of the paper is as follows. In Section 2, we present some preliminary knowledge on heat equations. In Section 3, we show
 the quantitative global Carleman estimate and the corresponding observability inequality in Theorem \ref{th1}. In Section 4, we derive the null controllability for linear heat equations and obtain the existence of a regular control function, including the proof of Theorem \ref{th2}.
 In Section 5, we obtain the observability inequality for the one-dimensional heat equation, including the proofs of Theorem \ref{th3} and Corollary \ref{cor1}.
The Appendix includes some useful regularity estimates for heat equations, needed for obtaining the regularity estimates of control functions.
 Throughout the paper, the letters $C$, $c$ or $C_i$ denote positive constants that do not depend on $V$ or the solution $y$, and may vary from line to line.

 \textbf{Acknowledgements.} Jiuyi Zhu is partially supported by NSF DMS-2154506.  Jinping Zhuge is partially supported by NNSF of China (No. 12494541, 12288201, 12471115). 

\section{Preliminaries}
In this section, we present some basic knowledge on the well-posedness and energy estimates of solutions of parabolic equations.
Let $S_T=L^2(0, T; H^1_0(\Omega))\cap H^1(0, T; H^{-1}(\Omega))$ and
\begin{equation}
    \| y \|_{S_T}^2: = \int_0^T \big( \| y(\cdot,t) \|_{H^1(\Omega)}^2 + \| y_t(\cdot,t) \|_{H^{-1}(\Omega)}^2 \big) dt.
\end{equation}
The following embedding is well-known
\begin{align}
S_T\hookrightarrow C(0, T; L^2(\Omega)).
\label{imbed}
\end{align}

Consider the inhomogeneous heat equation,
\begin{equation}
    \left\{
    \begin{aligned}
        y_t-\Delta y+V(x, t)y &= F \quad  &\text{in }& \ Q_T, \\
        y &=0  \quad &\text{on }& \ \Sigma_T, \\
        y( \cdot,0) &= y_0    &\text{on }& \  \Omega.
    \end{aligned}
    \right.
    \label{basic-in}
\end{equation}
The following well-posedness and energy estimate are classical (see \cite{Ev10} for a proof).  Let $V_- = \min \{ V, 0 \}$.
\begin{lemma}\label{lemm-11}
Let $V(x, t)\in L^\infty(Q_T)$, $F\in  L^2(Q_T)$ and $y_0\in L^2(\Omega)$. The Cauchy problem (\ref{basic-in}) admits a unique weak solution
$y\in S_T$. Furthermore, there exists $C(\Omega)$ such that
\begin{align}
\int_0^T \| y(\cdot,t) \|_{H^1(\Omega)}^2 dt + \max_{0\leq t\leq T}\| y(\cdot,t) \|_{L^2(\Omega)}^2 \leq C e^{CT \|V_-\|_{\infty}} (\|y_0\|_{L^2(\Omega)}^2+ \|F\|_{L^2(Q_T)}^2 ).
\label{reg-first}
\end{align}
\end{lemma}

Next, we consider the adjoint heat equation for (\ref{nonlinear-core}), which is given as
\begin{equation}
    \left\{
    \begin{aligned}
        -q_t-\Delta q+V(x, t)q &= 0 \quad  &\text{in }& Q_T, \\
        q &=0  \quad &\text{on }& \ \Sigma_T, \\
        q( \cdot, T) &= q_T    &\text{on }& \  \Omega.
    \end{aligned}
    \right.
    \label{adjoint}
\end{equation}
The adjoint heat equation (\ref{adjoint}) with a bounded potential has the following dissipativity in time of energy.
\begin{lemma}
Let $V(x, t)\in L^\infty(Q_T)$ and $q$ a solution of \eqref{adjoint}. Then there exists $C>0$ such that for any $0<t_1\leq t_2\leq T$,
\begin{align}
  \|q(\cdot,t_1)\|_{L^2(\Omega)}\leq Ce^{C{(t_2 -t_1)} \|V_-\|_{\infty}} \|q( \cdot, t_2)\|_{L^2(\Omega)}.
\label{dissipa}
\end{align}
\label{lem-dissipa}
\end{lemma}

The proof of the Lemma \ref{lem-dissipa} is standard. Multiplying (\ref{adjoint}) by $q$ and integrating by parts over $\Omega$ lead to
\begin{align*}
-\frac{1}{2}\partial_t\int_{\Omega} q^2(\cdot,t)+\int_{\Omega} |\nabla q(\cdot,t)|^2 +\int_{\Omega} V(\cdot,t) q^2(\cdot,t)=0.
\end{align*}
Thus,
\begin{align*}
\partial_t \|q(\cdot,t)\|^2_{L^2(\Omega)}+2\|V_-(\cdot,t)\|_{L^\infty(\Omega)} \|q(\cdot,t)\|^2_{L^2(\Omega)}\geq 0.
\end{align*}
An application of the Gronwall's inequality concludes the proof.

\section{The quantitative global Carleman estimate}
We define a function space $Z=\{a(x,t)\in L^\infty(0, T; W^{1,\infty}(\Omega))| a_t \in  L^\infty(0, T; L^{\infty}(\Omega))\}$.
We show a quantitative Carleman estimate that takes the regularity of $V(x, t)\in Z $ into consideration. Especially, the Carleman estimate is quantitative in the sense that the lower bound of Carleman parameter $\tau$ depends explicitly on certain norms of $V$. See Lemma \ref{Carle-main} below for the statement.

The following lemma is standard, see \cite{FI96}.
\begin{lemma}
Let $\Omega'\Subset \Omega$ be a nonempty open subset. Then there exists $\xi(x)\in C^2(\overline{\Omega})$ such that $\xi(x)$ is positive in $\Omega$ and vanishes on $\partial\Omega$. Moreover, $|\nabla \xi|>0$ in $\overline{\Omega\backslash \Omega'}$.
\label{aux-s}
\end{lemma}

Let $\Omega'\Subset \omega\Subset \Omega$. We introduce the following weight functions:
\begin{align}
\beta(x, t)=\frac{e^{2\lambda s\|\xi\|_{L^\infty}}- e^{\lambda(s\|\xi\|_{L^\infty}+\xi(x))}}{t(T-t) },
\end{align}
and
\begin{align}
\eta(x, t)=\frac{ e^{\lambda(s\|\xi\|_{L^\infty}+\xi(x))}}{t(T-t) }
\end{align}
for some large positive constants $s$ and $\lambda$. 

The main Carleman estimate is stated as follows.

\begin{lemma}
There exist positive constants $\lambda_0$, $$\tau_0=C(T+T^2+T^2\|V\|^{1/2}_{\infty}+T^2\|\nabla V\|^{1/2}_{\infty}+T^2\|\partial_t V\|^{1/3}_{\infty})
$$ 
and $C_1 = C_1(\Omega, \omega)$ such that,
for all $\lambda\geq \lambda_0$, $\tau\geq \tau_0$, we have
\begin{align}
C_1 \iint_{Q_T} e^{-2\tau \beta}|w_t+\Delta w&+V(x, t)w|^2 \,dxdt+ C_1\tau^3\lambda^4 \iint_{ \omega\times (0, T)}e^{-2\tau \beta} \eta^3 w^2 \,dxdt \nonumber \\
&\geq \tau^3\lambda^4 \iint_{Q_T}e^{-2\tau \beta} \eta^3 w^2 \,dxdt+ \tau\lambda^2 \iint_{Q_T}e^{-2\tau \beta} \eta |\nabla w|^2 \,dxdt \nonumber \\
&+\tau^{-1} \iint_{Q_T} e^{-2\tau \beta}\eta^{-1} (|\Delta w|^2 + |\partial_t w|^2) \,dxdt
\label{Car-est}
\end{align}
for all $w\in C^2(\overline{\Omega})$ with $w=0$ on $\Sigma_T$.
\label{Carle-main}
\end{lemma}

\begin{proof}
The proof is adapted from \cite{FI96} by Fursikov and Imanuvilov. The new idea is to take advantage of the regular potential $V(x, t)\in Z$ in the Carleman estimate to find the better lower bound of the Carleman parameter $\tau$ in terms of the norms of $V(x, t)$. 

To begin with,
we introduce the conjugate operator $\varphi= e^{-\tau \beta} w$. Let
\begin{align}
f=(\partial_t + \Delta + V) w.
\label{origi-1}
\end{align}
Then
\begin{align}
e^{-\tau \beta}(\partial_t + \Delta + V) ( e^{\tau \beta} \varphi)=e^{-\tau \beta} f.
\label{origi-2}
\end{align}
Using the observations
\begin{equation}\label{Dbeta}
    \nabla \beta = - \nabla \eta = - \lambda \eta \nabla \xi, \quad \text{and} \quad \Delta \beta = -\lambda^2 \eta |\nabla \xi|^2 - \lambda \eta \Delta \xi,
\end{equation}
we can simplify the equation (\ref{origi-2}) to
\begin{align}
A \varphi+B \varphi=F_{\tau, \lambda},
\label{equation-key}
\end{align}
where
\begin{align}
A\varphi =-2\tau \lambda^2 \eta \varphi |\nabla \xi|^2 - 2\tau \lambda \eta \nabla \xi \cdot \nabla \varphi+\partial_t\varphi,
\label{identi-1}
\end{align}
\begin{align}
B \varphi =\tau^2 \lambda^2 \eta^2\varphi |\nabla \xi|^2 +\Delta \varphi+ \tau \varphi \partial_t \beta+V(t,x)\varphi,
\label{identi-2}
\end{align}
and
\begin{align*}
F_{\tau, \lambda}=e^{-\tau \beta} f-\tau \lambda^2 \eta \varphi |\nabla \xi|^2 +\lambda\tau \eta\varphi \Delta\xi .
\end{align*}
Let us denote each term  in $A\varphi$ and $B\varphi$ in order as  $A_i\varphi$ and $B_j\varphi$. Hence
\begin{align*}
A\varphi=A_1\varphi+A_2\varphi+A_3\varphi,
\end{align*}
\begin{align*}
B\varphi=B_1\varphi+B_2\varphi+B_3\varphi {+ B_4 \varphi}.
\end{align*}
We compute the $L^2$ norm for the terms in the equation (\ref{equation-key}) to get
\begin{align}
\|A\varphi\|^2+ \|B\varphi\|^2 +2 \langle A\varphi, B\varphi \rangle =\| F_{\tau, \lambda}  \|^2.
\label{equation-red}
\end{align}
Note that
\begin{align*}
\langle A\varphi, B\varphi \rangle=\sum_{j=1}^4 \sum^3_{i=1} \langle A_i\varphi, B_j\varphi \rangle.
\end{align*}

To obtain an inequality from (\ref{equation-red}), we need to
compute the inner product $\langle A\varphi, B\varphi \rangle$. There are 12 terms in the inner product $\langle A\varphi, B\varphi \rangle$. Many efforts below are devoted to estimating these 12 terms. 

We start with the terms in $\langle A \varphi, B_1\varphi \rangle$.
First, note that
\begin{equation}\label{AB1-1}
    \langle A_1\varphi, B_1\varphi \rangle = -2\tau^3 \lambda^4 \iint_{Q_T} \eta^3 |\nabla \xi|^4 \varphi^2 \,dxdt=: N_1.
\end{equation}
Next, performing the integration by parts leads to
\begin{align}
\langle A_2\varphi, B_1\varphi \rangle&= -\tau^3 \lambda^3 \iint_{Q_T} \eta^3 \nabla \xi\cdot \nabla \varphi^2 |\nabla \xi|^2 \nonumber \\
&=3\tau^3\lambda^4 \iint_{Q_T} \eta^3 |\nabla \xi|^4 \varphi^2 +\tau^3 \lambda^3 \iint_{Q_T} \eta^3 \Delta \xi |\nabla \xi|^2 \varphi^2 \nonumber \\
&+2\tau^3 \lambda^3 \iint_{Q_T} \eta^3 D_j\xi D_{ij}\xi D_i\xi \varphi^2 \nonumber \\
&=M_1+M_2+M_3,
\label{AB1-2}
\end{align}
where $M_i$ represents the terms in the second equality and the summations over $i$ and $j$ are understood in the context of Einstein notations. Since $\|\xi\|_{C^2(\Omega)}\leq C$, the sum of $M_2$ and $M_3$ can be bounded as
\begin{align}
|M_2+M_3|\leq C\tau^3 \lambda^3 \iint_{Q_T} \eta^3 \varphi^2.
\label{AB1-3}
\end{align}
Now, we combine $M_1$ with $N_1$ to obtain
\begin{align}
N_1+M_1& =\tau^3\lambda^4 \iint_{Q_T} \eta^3 |\nabla \xi|^4 \varphi^2 \nonumber \\
& =\tau^3\lambda^4 \iint_{Q_T\backslash \Omega'\times (0, \ T)}  \eta^3 |\nabla \xi|^4 \varphi^2 +\tau^3\lambda^4 \iint_{ \Omega'\times (0, \ T)}  \eta^3 |\nabla \xi|^4 \varphi^2 \nonumber \\
&\geq c\tau^3\lambda^4 \iint_{Q_T} \eta^3  \varphi^2 -C  \tau^3\lambda^4 \iint_{\Omega'\times(0, \ T) } \eta^3  \varphi^2,
\label{AB1-4}
\end{align}
where we have used the property of $\xi$ in Lemma \ref{aux-s} in the last inequality.

We consider $\langle A_3\varphi, B_1\varphi \rangle$. A direct calculation shows that
\begin{align*}
\langle A_3\varphi, B_1\varphi \rangle&=\frac{1}{2}\tau^2\lambda^2 \iint_{Q_T}  \eta^2 |\nabla \xi|^2 \partial_t \varphi^2\nonumber \\
&=-\tau^2\lambda^2 \iint_{Q_T} \partial_t \eta \eta |\nabla \xi|^2 \varphi^2,
\end{align*}
where we used the fact  $\eta \varphi \to 0$ as $ t \to 0$ or $t \to T$. To see this fact, note that $\varphi = e^{-\tau \beta} w$, and both $\beta$ and $\eta$ behave like $t^{-1}$ as $t\to 0$ and $(T-t)^{-1}$ as $t\to T$. Thus the exponential decay $\varphi$ is much stronger than the polynomial growth of $\eta$ when $t\to 0$ or $T$, which leads to the above fact. Now, by
\begin{align*}
|\partial_t \eta| = \Big| \frac{(T-2t) e^{\lambda (s\|\xi\|_{L^\infty}+\xi)}  } {(t(T-t))^2} \Big| \leq T \eta^2,
\end{align*}
we have
\begin{align}
| \langle A_3\varphi, B_1\varphi \rangle | \leq C\tau^2\lambda^2 T \iint_{Q_T}  \eta^3 |\nabla \xi|^2  \varphi^2.
\label{AB1-5}
\end{align}
Combining (\ref{AB1-1}) through (\ref{AB1-5}) and by assuming that $\lambda$ is a large constant and $\tau>CT$, we have
\begin{align}
\langle A\varphi, B_1\varphi \rangle \geq C\tau^3\lambda^4  \iint_{Q_T}  \eta^3   \varphi^2-C\tau^3\lambda^4  \iint_{\Omega'\times(0, T)}  \eta^3   \varphi^2.
\label{innAB-1}
\end{align}

%We have done the calculations for every term in $\langle A\varphi, B_1\varphi \rangle$. 

Next, let us consider the terms in the inner product $\langle A\varphi, B_2\varphi \rangle$.
%We  consider the term $\langle A_1\varphi, B_2\varphi \rangle$. 
Performing the integration by parts shows that
\begin{align}
\langle A_1\varphi, B_2\varphi \rangle&=2\tau \lambda^2 \iint_{Q_T}  \eta  |\nabla \xi|^2 |\nabla \varphi|^2+2\tau \lambda^3 \iint_{Q_T}  \eta  |\nabla \xi|^2 \nabla \xi\cdot \nabla \varphi \varphi \nonumber \\
& +4\tau \lambda^2 \iint_{Q_T}  \eta  D_i\xi D_{ij} \xi D_j \varphi \varphi \nonumber \\
&=R_1+R_2+R_3.
\label{three-R}
\end{align}
We will keep $R_1$ and estimate $R_2$ and $R_3$. The Cauchy-Schwarz inequality shows that
\begin{align}
|R_2|
%& =2\tau \lambda^3 \iint_{Q_T}  \eta  |\nabla \xi|^2 \nabla \xi\cdot \nabla \varphi \varphi \nonumber\\
\leq C\tau^2 \lambda^4 \iint_{Q_T}  \eta^2  \varphi^2 + C \lambda^2 \iint_{Q_T}  |\nabla  \varphi|^2,
\label{three-R1}
\end{align}
where we used the fact $|\nabla \xi|\leq C$ in $Q_T$. The same reasoning also yields that
\begin{align}
|R_3| 
%& = 4\tau \lambda^2 \iint_{Q_T}  \eta  D_i\xi D_{ij} \xi D_j \varphi \varphi \nonumber\\
 \leq C{\tau^2}  \lambda^4 \iint_{Q_T}  {\eta^2}  \varphi^2 +  \iint_{Q_T}  |\nabla  \varphi|^2.
\label{three-R2}
\end{align}
Since $\eta\geq C T^{-2}$, by taking $\tau>CT^2$ for sufficiently large $C$, we obtain from (\ref{three-R})-(\ref{three-R2}) that
\begin{align}
\langle A_1\varphi, B_2\varphi \rangle &\geq {2\tau \lambda^2} \iint_{Q_T} \eta  |\nabla \xi|^2 |\nabla \varphi|^2- C\tau^2 \lambda^4 \iint_{Q_T} \eta^2  \varphi^2 \nonumber \\
&-C\iint_{Q_T} ({\lambda^2+1}) |\nabla  \varphi|^2.
\label{AB-21}
\end{align}

We continue to estimate $\langle A_2 \varphi, B_2\varphi \rangle$. Applying the integration by parts gives that
\begin{align}
\langle A_2\varphi, B_2\varphi \rangle &=-2\tau \lambda \iint_{Q_T} \eta  \nabla \varphi \cdot \nabla \xi \Delta \varphi\nonumber \\
&=-2\tau \lambda \iint_{\Sigma_T} \eta \frac{\partial \xi}{\partial \nu} | \frac{\partial \varphi}{\partial \nu}|^2 +
2\tau \lambda^2 \iint_{Q_T} \eta  |\nabla \xi\cdot \nabla \varphi|^2 \nonumber \\
&+2 \tau \lambda  \iint_{Q_T} \eta  D_i\varphi D_{ij}\xi D_j \varphi + 2 \tau \lambda  \iint_{Q_T} \eta  D_i\xi D_{ij}\varphi D_j \varphi \nonumber \\
&=S_1+S_2+S_3+S_4, \label{ss-1}
\end{align}
{where $\nu=\frac{\nabla \varphi}{|\nabla \varphi|}$ is the unit outer normal and $|\nabla \varphi| = |\partial \varphi/\partial \nu|$, due to $\varphi=0$ on $\Sigma_T$. Note that $\partial \xi /\partial \nu<0$ on $\Sigma_T$ because $\xi=0$ on $\Sigma_T$ and $\xi>0$ in $Q_T$. Thus $S_1 \geq 0$. Obviously, we also have $S_2 \geq 0$.}
For the term $S_3$,
we can estimate as follows,
\begin{align}
|S_3| \leq C\tau \lambda \iint_{Q_T}\eta  |\nabla \varphi|^2. \label{ss-2}
\end{align}
For the term $S_4$, the integration by parts gives that
\begin{align}
S_4 &= \tau\lambda \iint_{Q_T}  \eta \nabla \xi\cdot \nabla |\nabla \varphi|^2 \nonumber \\
& =\tau\lambda \iint_{\Sigma_T} \eta  |\nabla \varphi|^2\frac{\partial \xi}{\partial \nu}-\tau\lambda^2 \iint_{Q_T} \eta |\nabla \xi|^2 |\nabla \varphi|^2
-\tau\lambda \iint_{Q_T}\eta \Delta \xi |\nabla \varphi|^2 \nonumber \\
&=S_{41} + S_{42}+ S_{43}. \label{ss-3}
\end{align}
{Since $|\nabla \varphi| = |\partial \varphi/\partial \nu|$, we have
\begin{equation}
    S_1 + S_{41} = -\tau\lambda \iint_{\Sigma_T} \eta  |\nabla \varphi|^2\frac{\partial \xi}{\partial \nu} \geq 0.
\end{equation} }
{We will keep $S_{42}$} and estimate $S_{43}$ as
\begin{align}
|S_{43}| \leq C\tau\lambda \iint_{Q_T} \eta  |\nabla \varphi|^2. \label{ss-4}
\end{align}
The combination of (\ref{ss-1})-(\ref{ss-4}) gives
\begin{align}
\langle A_2\varphi, B_2\varphi \rangle & \geq -\tau\lambda^2 \iint_{Q_T} \eta |\nabla \xi|^2 |\nabla \varphi|^2-C\tau \lambda
 \iint_{Q_T} \eta |\nabla \varphi|^2.
 \label{AB-22}
\end{align}

Let us turn to the last term $\langle A_3\varphi, B_2\varphi \rangle$ in $\langle A\varphi, B_2\varphi \rangle$. 
{For a similar reason as before, we can see $|\nabla \varphi|\to 0$ as  $t\to 0$ or $t\to T$.}
%Note that $\beta\to \infty$  as $t\to 0$ or $t\to T$. Recall that  $\varphi= e^{-\tau \beta}w$. 
Then 
\begin{align*}
\langle A_3\varphi, B_2\varphi \rangle =-\iint_{Q_T} \nabla \partial_t \varphi \cdot \nabla  \varphi \nonumber 
=-\frac{1}{2} \iint_{Q_T}  \partial_t |\nabla \varphi|^2 =0.
\end{align*}
Together with the previous estimates (\ref{AB-21}) and (\ref{AB-22}), we have
\begin{align}
\langle A\varphi, B_2\varphi \rangle &\geq {\tau\lambda^2}  \iint_{Q_T} \eta |\nabla \xi|^2 |\nabla \varphi|^2- C \tau^2 \lambda^4  \iint_{Q_T} \eta^2 | \varphi|^2 \nonumber \\
&-C \iint_{Q_T} (\lambda^2+ \tau\lambda \eta) |\nabla \varphi|^2
\label{innAB-2}
\end{align}
for $\lambda$ large and $\tau\geq CT^2$. From the property of $\xi$ that {$|\nabla \xi| \geq c>0$ in $Q_T \setminus \Omega'\times (0,T)$}, we can manipulate the right-hand side of \eqref{innAB-2} as
\begin{align}
\langle A\varphi, B_2\varphi \rangle& \geq \tau\lambda^2 \iint_{Q_T\backslash \Omega'\times (0, T)} \eta |\nabla \xi|^2 |\nabla \varphi|^2 +\tau\lambda^2 \iint_{ \Omega'\times (0, T)} \eta |\nabla \xi|^2 |\nabla \varphi|^2\nonumber \\ &- C \tau^2 \lambda^4  \iint_{Q_T} \eta^2 | \varphi|^2
-C \iint_{{Q_T}\backslash\Omega'\times (0, T)} (\lambda^2+ \tau\lambda \eta) |\nabla \varphi|^2 \nonumber \\ &-\iint_{ \Omega'\times (0, T) } C(\lambda^2+ \tau\lambda \eta) |\nabla \varphi|^2 \nonumber \\
&\geq {\frac12 c^2 \tau\lambda^2} \iint_{Q_T} \eta |\nabla \varphi|^2 - C\tau\lambda^2  \iint_{ \Omega'\times (0, T)}  \eta |\nabla \varphi|^2 \nonumber \\& -C\tau^2\lambda^4 \iint_{{Q_T}}  \eta^2 | \varphi|^2,
\label{innAB-2}
\end{align}
where we have used the fact $ \eta> C T^{-2}$ and $\tau>C' T^2$ (with sufficiently large $C'$) to guarantee $C(\lambda^2 + \tau \lambda \eta) \le \frac12 c^2 \tau \lambda^2 \eta$. This completes the estimate of $\langle A\varphi, B_2\varphi \rangle$.

Next we consider $\langle A\varphi, B_3\varphi \rangle  $. To this end, we first note that
\begin{align}
|\partial_t \beta| =\Big| \frac{(T-2t)}{t^2(T-t)^2} (e^{2\lambda s\|\xi\|_{L^\infty}}- e^{\lambda (s\|\xi\|_{L^\infty}+\xi)}) \Big| \leq CT \eta^2.
\label{beta-1}
\end{align}
For the first term $\langle A_1\varphi, B_3\varphi \rangle$, we have
\begin{align}
|\langle A_1\varphi, B_3\varphi \rangle| &\le 2 \tau^2\lambda^2 \iint_{Q_T} \eta |\nabla \xi|^2 \varphi^2 | \partial_t\beta| \nonumber \\
&\leq CT\tau^2\lambda^2 \iint_{Q_T} \eta^3  \varphi^2,
\label{AB3}
\end{align}
where we have used (\ref{beta-1}).
For the second term $\langle A_2\varphi, B_3\varphi \rangle$, applying the integration by parts, we have
\begin{align*}
\langle A_2\varphi, B_3\varphi \rangle&=- \tau^2\lambda \iint_{Q_T} \eta \nabla \xi \cdot\nabla \varphi^2 \partial_t\beta \nonumber \\
&= \tau^2\lambda^2 \iint_{Q_T} \eta |\nabla \xi|^2  \varphi^2 \partial_t\beta + \tau^2\lambda \iint_{Q_T} \eta \Delta \xi  \varphi^2 \partial_t\beta \nonumber \\
&+\tau^2\lambda  \iint_{Q_T} \eta \nabla \xi \cdot \nabla \partial_t \beta  \varphi^2.
\end{align*}
In view of (\ref{beta-1}), all the integrals in the lase equality are bounded by $C\tau^2\lambda^2 T  \iint_{Q_T} \eta^3  \varphi^2$.
% We can estimate all of them in the last inequality by $\tau^2\lambda^2 T^2  \iint_{Q_T} \eta^3  \varphi^2$ by the fact of (\ref{beta-1}).
Thus, we have
\begin{align}
|\langle A_2\varphi, B_3\varphi \rangle| \leq  C\tau^2\lambda^2 T  \iint_{Q_T} \eta^3  \varphi^2.
\end{align}
Next, we estimate the third term $\langle A_3 \varphi, B_3\varphi \rangle$ in $\langle A \varphi, B_3\varphi \rangle$ as
\begin{align}
|\langle A_3\varphi, B_3\varphi \rangle| &= \frac{1}{2}\tau \Big| \iint_{Q_T} \partial_t \beta \partial_t \varphi^2 \Big| \nonumber \\
&=\frac12 \tau \Big| \iint_{Q_T} \partial_{tt} \beta  \varphi^2 \Big| \nonumber \\
&\leq C\tau T^2 \iint_{Q_T}  \eta^3 \varphi^2,
\label{AB3-1}
\end{align}
where we used the fact $|\partial_{tt} \beta|\leq C\eta^3 T^2$ and $ \eta \varphi\to 0$ as $t\to 0$ or $t\to T$.

Thanks to the fact that $\lambda>C$ and $\tau>C(T+T^2)$, the combination of the estimates (\ref{AB3})-(\ref{AB3-1}) yields that
\begin{align}
\langle A\varphi, B_3\varphi \rangle \geq -C\tau^3\lambda^2 \iint_{Q_T}  \eta^3 \varphi^2.
\label{innAB-3}
\end{align}

Finally, we compute $\langle A\varphi, B_4\varphi \rangle$. The assumption $V(x,t)\in Z$ implies that $\| V\|_\infty, \|\nabla V\|_\infty$ and $\| \partial_t V\|_\infty$ are all bounded. We can simply estimate the first term in  $\langle A\varphi, B_4\varphi \rangle$ as
\begin{align}
|\langle A_1\varphi, B_4\varphi \rangle | \leq C\tau\lambda^2 \|V\|_{\infty}\iint_{{Q_T}} \eta \varphi^2.
\label{AB4-1}
\end{align}
For the second term $\langle A_2\varphi, B_4\varphi \rangle$,
the integration by parts shows that
\begin{align}
|\langle A_2\varphi, B_4\varphi \rangle | &= \tau \lambda \Big| \iint_{{Q_T}} \eta V \nabla \xi \cdot \nabla \varphi^2 \Big| \nonumber \\
&\le \tau \lambda^2 \Big| \iint_{{Q_T}} \eta V |\nabla \xi|^2  \varphi^2 \Big| +\tau \lambda \Big| \iint_{{Q_T}} \eta V \Delta \xi  \varphi^2 \Big|
+\tau \lambda \Big| \iint_{{Q_T}} \eta  \nabla \xi \cdot \nabla V  \varphi^2 \Big| \nonumber \\
&\leq C\tau \lambda^2 \|V\|_{\infty} \iint_{{Q_T}} \eta \varphi^2+ C\tau\lambda  \|\nabla V\|_{\infty}\iint_{{Q_T}} \eta \varphi^2,
\label{AB4-2}
\end{align}
where we used $\|\xi\|_{C^2(Q_T)}\leq C$. For the third term in $\langle A_3\varphi, B_4\varphi \rangle$, we bound it as
\begin{align}
|\langle A_3\varphi, B_4\varphi \rangle| &=\frac{1}{2} \Big| \iint_{{Q_T}} V \partial_t \varphi^2 \Big| \nonumber \\
&=\frac{1}{2} \Big| \iint_{{Q_T}} \partial_t V \varphi^2 \Big| \nonumber \\
&\leq C \|\partial_t V\|_{\infty} \iint_{{Q_T}} \varphi^2,
\label{AB4-3}
\end{align}
where we have used $\varphi \to 0$ as $t \to 0$ or $t\to T$.
Together with (\ref{AB4-1}) and (\ref{AB4-2}), and by $\eta \geq CT^{-2}$, we arrive at
\begin{align}
|\langle A\varphi, B_4\varphi \rangle | \leq C( \tau \lambda^2  \|V\|_{\infty}+\tau \lambda \|\nabla V\|_{\infty}+
T^2\|\partial_t V\|_{\infty})\iint_{Q_T} \eta \varphi^2.
\label{innAB-4}
\end{align}

We have assumed that  $\tau>C(T+T^2)$ in the previous estimates.
 Taking account of the estimates (\ref{innAB-1}), (\ref{innAB-2}), (\ref{innAB-3}), and  (\ref{innAB-4}) and choosing $\lambda, s$ large and
 \begin{align}
 \tau \geq C(T+T^2+T^2\|V\|^{\frac{1}{2}}_{\infty}+T^2\|\nabla V\|^{\frac{1}{2}}_{\infty}+T^2\|\partial_t V\|^{\frac{1}{3}}_{\infty}),
 \label{tau-a}
 \end{align}
we have
\begin{align*}
\langle A\varphi, B\varphi \rangle &\geq c\tau\lambda^2 \iint_{{Q_T}} \eta |\nabla \varphi|^2 +c\tau^3\lambda^4 \iint_{{Q_T}} \eta^3  \varphi^2 \nonumber \\
&-C \tau\lambda^2 \iint_{\Omega' \times (0, T)} \eta |\nabla \varphi|^2 - C \tau^3\lambda^4 \iint_{\Omega' \times (0, T)} \eta^3  \varphi^2.
\end{align*}
From the identity (\ref{equation-red}), we get that
\begin{align}
& \|A\varphi\|^2+ \|B \varphi\|^2+c \tau\lambda^2 \iint_{{Q_T}} \eta |\nabla \varphi|^2 +c \tau^3\lambda^4 \iint_{{Q_T}} \eta^3  \varphi^2 \nonumber \\
&\leq \|F_{\tau, \lambda}\|^2+
C\iint_{\Omega'\times (0, T)} \tau\lambda^2\eta |\nabla \varphi|^2 \nonumber +C\iint_{\Omega'\times (0, T)}\tau^3\lambda^4 \eta^3  \varphi^2 \nonumber \\
&\leq C \iint_{{Q_T}} e^{-2\tau \beta} f^2+ C\tau^2\lambda^4 \iint_{{Q_T}} \eta^2  \varphi^2 +
C \tau\lambda^2 \iint_{ \Omega'\times (0, T)} \eta |\nabla \varphi|^2 +\tau^3\lambda^4 \iint_{ \Omega'\times (0, T)} \eta^3  \varphi^2.
\end{align}
By assuming $\tau \geq CT^2$ such that $c\tau \eta \geq 2C$, the second integral on the right-hand side can be absorbed to the left. It follows that
\begin{align}
&\|A\varphi\|^2+ \|B \varphi\|^2+\iint_{{Q_T}} c\tau\lambda^2\eta |\nabla \varphi|^2 +c\tau^3\lambda^4 \eta^3  \varphi^2
\nonumber \\ 
& \qquad \leq  C\iint_{{Q_T}} e^{-2\tau \beta} f^2+ C\iint_{ \Omega'\times (0, T)} \tau\lambda^2\eta |\nabla \varphi|^2 + C\iint_{ \Omega'\times (0, T)}\tau^3\lambda^4 \eta^3  \varphi^2.
\label{inco-left}
\end{align}

Next, we want to incorporate $\iint_{ \Omega'\times (0, T)} \tau\lambda^2\eta |\nabla \varphi|^2$ into the left-hand side of (\ref{inco-left}). To this end, we need to add $|\Delta \varphi|$ to the Carleman estimates (\ref{inco-left}).
By the identity (\ref{identi-2}) and the triangle inequality, thanks to (\ref{beta-1}), we have
\begin{align}
c\tau^{-1} \iint_{Q_T} \eta^{-1} |\Delta \varphi|^2& \leq \tau^3 \lambda^4 \iint_{Q_T}  \eta^3  \varphi^2+ \tau T^2 \iint_{Q_T}  \eta^3 \varphi^2 \nonumber \\
&+\tau^{-1} \iint_{Q_T}  \eta^{-1} |V|^2 \varphi^2 +\tau^{-1}\|\eta^{-\frac{1}{2}} B \varphi\|^2 \nonumber \\
&\leq C\tau^3 \lambda^4 \iint_{Q_T}  \eta^3 \varphi^2+\| B \varphi\|^2,
\end{align}
where we have used the fact  $ \|V\|_{\infty}^2\eta^{-1}\tau^{-1}\leq \tau^3\eta^3$, due to the assumption of  $\tau$ in (\ref{tau-a}),  and $\eta^{-1}\leq CT^2$. Similarly, the identity (\ref{identi-1}) and the assumption of $\xi$ lead to
\begin{align}
c\tau^{-1} \iint_{Q_T} \eta^{-1} |\partial_t \varphi|^2\leq \tau \lambda^4 \iint_{Q_T} \eta \varphi^2+\tau\lambda^2 \iint_{Q_T} \eta |\nabla \varphi|^2+ \tau^{-1} \| \eta^{-\frac{1}{2}} A\varphi\|^2.
\end{align}
Therefore, from (\ref{inco-left}), we deduce that
\begin{align}
 & \iint_{Q_T}  \tau^{-1}\eta^{-1}(\partial_t \varphi+ |\Delta \varphi|)^2 +\tau\lambda^2 \iint_{Q_T}  \eta |\nabla \varphi|^2
 +\tau^3\lambda^4  \iint_{Q_T}  \eta^3  \varphi^2\nonumber \\
 &\qquad \leq C  \iint_{Q_T}  e^{-2\tau \beta} f^2+ \tau\lambda^2 \iint_{\Omega' \times (0, T)} \eta |\nabla \varphi|^2 +\tau^3\lambda^4 \iint_{\Omega' \times (0, T)} \eta^3  \varphi^2.
 \label{impor-cal}
\end{align}

We want to control the second integral  on the right-hand side in the last inequality. Recall that $ \Omega'\Subset\omega\Subset \Omega$. We introduce a cutoff function
\begin{align*}
0\leq \phi\leq 1, \quad \phi\in C_0^2(\omega),\quad \phi=1 \ \mbox{in} \ \Omega'.
\end{align*}
It is obvious that
\begin{align*}
\tau\lambda^2 \iint_{ \Omega'\times (0, T)} \eta |\nabla \varphi|^2\leq \tau\lambda^2 \iint_{ \omega\times (0, T)}\phi \eta |\nabla \varphi|^2.
\end{align*}
The applications of integration by parts show that
\begin{align*}
\iint_{ \omega\times (0, T)}\phi \eta |\nabla \varphi|^2&=- \iint_{\omega\times (0, T)}\eta \nabla \phi\cdot \nabla \varphi \varphi-\lambda  \iint_{\omega\times (0, T)}\phi \eta \nabla \xi \cdot \nabla \varphi \varphi
-\iint_{\omega\times (0, T)}\phi \eta \Delta \varphi \varphi \nonumber \\
&=\frac{1}{2} \iint_{\omega\times (0, T)}\nabla( \eta \nabla \phi) \varphi^2 +\frac{\lambda}{2} \iint_{\omega \times (0, T)}\nabla(\phi \eta \nabla \xi) \varphi^2-\iint_{\omega\times (0, T)}\phi \eta \Delta \varphi \varphi\nonumber \\
&\leq C\lambda^2 \iint_{\omega\times (0, T)}  \eta  \varphi^2+ \epsilon \tau^{-2}\lambda^{-2} \iint_{\omega\times (0, T)}  \eta^{-1} |\Delta \varphi|^2 \nonumber \\ &+C\epsilon^{-1} \tau^2\lambda^2 \iint_{\omega\times (0, T)}  \eta^3  \varphi^2,
\end{align*}
where we have used the Young's inequality in the last inequality with some small constant $\epsilon>0$ to be selected.
Therefore, we obtain
\begin{align*}
\tau\lambda^2 \iint_{ \Omega'\times (0, T)} \eta |\nabla \varphi|^2 &\leq \tau \lambda^4 \iint_{\omega\times (0, T)}  \eta  \varphi^2+ \epsilon \tau^{-1} \iint_{\omega\times (0, T)}  \eta^{-1} |\Delta \varphi|^2 \nonumber\\
&+C\epsilon^{-1} \tau^3\lambda^4 \iint_{\omega\times (0, T)}  \eta^3  \varphi^2 \nonumber\\
&\leq   C\epsilon^{-1} \tau^3\lambda^4 \iint_{\omega\times (0, T)}  \eta^3  \varphi^2 +\epsilon \tau^{-1} \iint_{\omega\times (0, T)}  \eta^{-1} |\Delta \varphi|^2,
\end{align*}
where we used the fact $\tau \eta>C$ again.
By selecting $\epsilon$ to be some small positive constant in the last inequality and take into consideration of  (\ref{impor-cal}), the integral $\epsilon \tau^{-1} \iint_{\omega\times (0, T) }  \eta^{-1} |\Delta \varphi|^2$ can be incorporated into the left-hand side of (\ref{impor-cal}). Then $\tau\lambda^2 \iint_{ \Omega'\times (0, T)} \eta |\nabla \varphi|^2$ can be absorbed.
Therefore, from (\ref{impor-cal}), we deduce that
\begin{align}
 \iint_{{Q_T}} \tau^{-1}\eta^{-1}(|\partial_t \varphi|+ |\Delta \varphi|)^2& +\tau\lambda^2 \iint_{{Q_T}} \eta |\nabla \varphi|^2
 +\tau^3\lambda^4  \iint_{{Q_T}} \eta^3  \varphi^2\nonumber \\
 &\leq C  \iint_{{Q_T}} e^{-2\tau \beta} f^2+ C \tau^3\lambda^4\iint_{ \omega \times (0, T)}  \eta^3  \varphi^2.
 \label{atlast-1}
\end{align}

It is remaining to return  to the Carleman estimate in terms of $w$. Recall that $\varphi= e^{-\tau \beta} w$.
Clearly, \eqref{atlast-1} yields
\begin{align}
\tau^3\lambda^4  \iint_{Q_T} e^{-2\tau \beta}  \eta^3  w^2\leq  C  \iint_{Q_T} e^{-2\tau \beta} f^2+ C\tau^3\lambda^4\iint_{\omega\times (0, T)}  e^{-2\tau \beta} \eta^3  w^2.
\label{final-1}
\end{align}
By \eqref{atlast-1} and
\begin{align*}
e^{-\tau \beta} \nabla w =\nabla \varphi -\tau \lambda e^{-\tau \beta} \eta\nabla \xi w,
\end{align*}
we have
\begin{align}
\tau \lambda^2 \iint_{Q_T} e^{-2\tau\beta} \eta |\nabla w|^2&\leq C  \iint_{Q_T} e^{-2\tau \beta} f^2+ C\tau^3\lambda^4\iint_{\omega\times (0, T)}  e^{-2\tau \beta} \eta^3  w^2.
 \label{atlast-2}
\end{align}
Thus, it follows from (\ref{atlast-1}) and (\ref{atlast-2}) that
\begin{align}
 \iint_{{Q_T}} \tau^{-1}\eta^{-1}(|\partial_t \varphi|+ |\Delta \varphi|)^2& +\tau\lambda^2 \iint_{{Q_T}} e^{-2\tau\beta} \eta |\nabla w|^2
 +\tau^3\lambda^4  \iint_{{Q_T}} e^{-2\tau\beta}\eta^3  w^2\nonumber \\
 &\leq C  \iint_{{Q_T}} e^{-2\tau \beta} f^2+ \tau^3\lambda^4\iint_{ \omega\times (0, T)}  e^{-2\tau\beta}\eta^3  w^2.
 \label{final-2}
\end{align}
Now notice that
\begin{align*}
\Delta \varphi= \tau^2  e^{-\tau\beta}|\nabla \beta|^2 w- \tau e^{-\tau\beta} \Delta \beta w-2\tau  e^{-\tau\beta} \nabla \beta\cdot\nabla w+ e^{-\tau\beta}\Delta w,
\end{align*}
and
\begin{equation*}
    \partial_t \varphi = -\tau e^{-\tau \beta} \partial_t \beta w + e^{-\tau \beta} \partial_t w.
\end{equation*}
Applying the triangle inequalities to the above identities and using \eqref{Dbeta} and \eqref{beta-1},
we get
\begin{align}
& \tau^{-1} \iint_{{Q_T}} e^{-2\tau\beta} \eta^{-1}( |\Delta w|^2 + |\partial_t w|^2) \nonumber \\
&\leq \tau^{-1} \iint_{{Q_T}} \eta^{-1} (|\Delta \varphi |^2 + |\partial_t \varphi|^2) +
C\tau^3 \lambda^4  \iint_{{Q_T}} e^{-2\tau\beta} \eta^3 w^2  \nonumber \\
&\qquad + C\tau \lambda^4  \iint_{{Q_T}} e^{-2\tau\beta} \eta w^2 +\tau \lambda^2 \iint_{{Q_T}} e^{-2\tau\beta} \eta |\nabla w|^2.
\label{final-3}
\end{align}
Taking the estimates (\ref{final-1}), (\ref{final-2}) and (\ref{final-3}) into consideration, we show that
\begin{align*}
\tau^3\lambda^4 \iint_{{Q_T}}e^{-2\tau \beta} \eta^3 w^2 \,dxdt&+ \tau\lambda^2 \iint_{Q}e^{-2\tau \beta} \eta |\nabla w|^2 \,dxdt
 \nonumber \\ &+\tau^{-1} \iint_{{Q_T}} e^{-2\tau \beta}\eta^{-1} (|\Delta w|^2+|\partial_t w|^2) \,dxdt  \nonumber \\
&\leq
C_1 \iint_{{Q_T}} e^{-2\tau \beta}f^2 \,dxdt+ C_1\tau^3\lambda^4 \iint_{\omega\times (0, T)}e^{-2\tau \beta} \eta^3 w^2 \,dxdt.
\end{align*}
This completes the proof of Lemma \ref{Carle-main}.
\end{proof}

Based on the Carleman estimates (\ref{Car-est}) in Lemma \ref{Carle-main}, we can show the observability inequality in Theorem \ref{th1}. For convenience, we let
\begin{equation}\label{defV}
   {\vertiii{V}  := \| V \|_\infty^{1/2} + \| \nabla V \|_{\infty}^{1/2} + \| \partial_t V \|_{\infty}^{1/3}.}
\end{equation}
\begin{proof}[Proof of Theorem \ref{th1}]
The proof should be standard. However, we present the details of the proof.
Since \eqref{adjoint} is the adjoint heat equation of \eqref{nonlinear-core},
to show (\ref{obser-lip}), it suffices to prove
\begin{align}
\|q(0)\|^2_{L^2(\Omega)}\leq e^{\hat{\tau}} \iint_{\omega\times (0, T)} q^2
\label{observability}
\end{align}
for the solution $q$ in (\ref{adjoint}),
where $\hat{\tau}=C(1+T^{-1}+T \|V\|_{\infty}+\vertiii{V})$.

Applying the Carleman estimates  (\ref{Car-est}) in Lemma \ref{Carle-main} to $w=q$ and using the equation $q_t+\Delta q - Vq = 0$, we have
\begin{align}
\iint_{Q_T}e^{-2\tau \beta} \eta^3  q^2 \leq C_1 \iint_{\omega\times (0, T)}e^{-2\tau \beta} \eta^3 q^2,
\label{carle-dedu}
\end{align}
for all $\tau\geq \tau_0= C(T+T^2+T^2 \vertiii{V})$. We choose $\tau=\tau_0$ in (\ref{carle-dedu}). Since $\eta \simeq \beta \simeq T^{-2}$ for $(x,t) \in \Omega \times (\frac{T}{3}, \frac{2T}{3})$, it is not difficult to see
\begin{equation}\label{exp-eta3}
    {e^{-2\tau \beta} \eta^3 \geq cT^{-6} \exp(-C(1+T^{-1} + \vertiii{V})) \qquad \text{in } \Omega \times (\frac{T}{3}, \frac{2T}{3}).}
\end{equation}
Next, we study the right-hand side of  (\ref{carle-dedu}). Since $\xi(x)>0$ in the open set $\omega$, we introduce $c_0:=\max_{\omega} \xi\leq \|\xi\|_{L^\infty}$. Furthermore, let $c_\ast:=e^{\lambda c_0}>1$, which is a fixed constant.  Then
\begin{align*}
e^{2\lambda s\|\xi\|_{L^\infty}}- e^{\lambda(s\|\xi\|_{L^\infty}+\xi(x))}&=e^{\lambda s\|\xi\|_{L^\infty}} ( e^{\lambda s\|\xi\|_{L^\infty}}- e^{\lambda \xi(x)} )\nonumber \\
&\geq c_\ast^s ( c_\ast^s -c_\ast  )
\nonumber \\
&=: c^\ast>0
\end{align*}
in the open set $\omega$ for some fixed constant $c^\ast$ and $s$ large.
Let $m=t(T-t)$. It is easy to see that $0\leq m\leq \frac14 T^2$. An elementary argument shows that $ e^{\frac{-2c^\ast \tau }{m} } m^{-3}$ is increasing if $m\leq \frac23 c^* \tau$. Thus, if we assume $\frac14 T^2 \le \frac23 c^* \tau$, we get
\begin{align}
e^{-2\tau \beta} \eta^3 &\leq C e^{\frac{-2c^\ast \tau }{m}} m^{-3} \le Ce^{ \frac{-8c^\ast \tau }{T^2}} (\frac14 T^2)^{-3} \nonumber \\
&\leq CT^{-6} \exp(-c(1+T^{-1}+ \vertiii{V} ))
\label{obser-12}
\end{align}
in $Q_T$ for some large constant $C$.
 Therefore, it follows from (\ref{carle-dedu}), (\ref{exp-eta3}) and \eqref{obser-12} that
\begin{align}
\iint_{\Omega\times (\frac{T}{3}, \frac{2T}{3})}  q^2 \leq \exp(C (1+T^{-1}+ \vertiii{V} )) \iint_{\omega\times (0, T)} q^2.
\label{observe-11}
\end{align}
By the dissipativity estimates (\ref{dissipa}) in Lemma \ref{lem-dissipa}, we have
\begin{align}
\|q(0)\|^2_{L^2(\Omega)}\leq CT^{-1} e^{CT \|V\|_{\infty} }\iint_{\Omega\times (\frac{T}{3}, \frac{2T}{3})}  q^2.
\label{observe-22}
\end{align}
%Let $\hat{\tau}=C(1+\frac{1}{T}+T \|V\|_{L^\infty(Q_T)}+\| V\|^{\frac{1}{2}}_{L^\infty(0, T;W^{1,\infty})}+\|\partial_t V\|^{\frac{1}{3}}_{L^\infty})$. 
Note that $CT^{-1} \le e^{1+CT^{-1}}$ for any $T>0$.
The combination of (\ref{observe-11}) and (\ref{observe-22}) gives that
\begin{align*}
\|q(0)\|^2_{L^2(\Omega)}\leq \exp(C(1+T^{-1}+T \|V\|_{\infty}+ \vertiii{V})) \iint_{\omega\times (0, T)} q^2.
\end{align*}
Thus, the conclusion in Theorem \ref{th1} is arrived.
\end{proof}
\section{ Null Controllability for linear heat equations }
A consequence of the observability inequality is the null controllability for the linear heat equation, following the so-called Hilbert Uniqueness Method introduced by J.L. Lions (see e.g., \cite{L88}). Since the proof now is standard, we refer to \cite[Theorem 2.44]{Cor07} and skip the proof.

\begin{proposition}\label{prop.HUM}
The observability inequality \eqref{observability} implies the null controllability of \eqref{linear-core-1} and the control function satisfies
\begin{align}\label{pro1}
\| h \|_{L^2({\omega\times (0, T)})}\leq \exp(C(1+T^{-1}+T \|V\|_{\infty}+ \vertiii{V})) \|y_0\|_{L^2(\Omega)},
\end{align}
where $\vertiii{V}$ is defined in \eqref{defV}.
\end{proposition}

{Now we are ready to prove the existence of a regular control function for the null controllability,  and estimate the norms of the control function and the trajectory  $y(x,t)$ }

Note that the control $h$ obtained by Proposition \ref{prop.HUM} is merely in $L^2$. In view of the Lipchitz regularity assumption on the potential, it is natural to ask for a more regular control. 
To this end, we borrow the idea from \cite{GP06} (also see \cite{FLM12}) which makes use of the parabolic regularizing effect in the case of Lipschitz potentials.

\begin{proof}[Proof of Theorem \ref{th2}]
Since $V$ is Lipschitz, the weak solution of (\ref{nonlinear-core}) satisfies the spatial $C^{2+\alpha}$ estimate for each time $t>0$. Precisely, by the energy estimate \eqref{lemm-11} and the classcial Schauder estimate for parabolic equations, for $0<t<\min\{T,1 \}$, we have
\begin{equation*}
    \|y(\cdot, t) \|_{C^{2+\alpha}(\overline{\Omega})}\leq C (t^{-1} + \| V \|_Z)^N  e^{Ct \|V\|_{\infty}} \|y_0\|_{L^2(\Omega)},
\end{equation*}
where $\| V \|_Z=\|V\|_\infty + \|\nabla V\|_\infty+ \|\partial_tV\|_\infty $ and $N$ (arising from rescaling and a bootstrap argument for regularity) is a constant depending only on $n$ and $\alpha$. 
{The proof of this estimate is similar to Lemma \ref{last-reg-lem} with careful attention to the dependence of $t$ and $\| V \|_Z$.}
%(Note: The above estimate for small $t$ must involve $\|V \|_Z$ even if $t = 1$. It must have a factor $t^{-N}$ that blows up as $t\to 0$, where the number $N$ may be determined by the heat kernel.)
%\textcolor{blue}{Kind of agree with the statement in this parenthesis, it depends on $\alpha$, $n$ and some iteration using heat kernel. Do we need to give a proof? It is not good to use ``may be'' or `` must''here}
By this parabolic regularizing effect, we may choose $t_0 = \min\{ \delta T, 1 \}$ with some small $\delta<1/3$ and consider $t_0$ as the initial time. In this case,
\begin{equation}
    \|y(\cdot, t_0) \|_{C^{2+\alpha}(\overline{\Omega})}\leq  \exp(C(1+T^{-1} + T \|V\|_{\infty} + \vertiii{V} )) \|y_0\|_{L^2(\Omega)},
\end{equation}
where we have used $T^{-1} \le \exp(T^{-1})$ and $\| V \|_Z \le \exp (1+T^{-1} + T\| V\|_\infty+ \vertiii{V}).$ Hence, without loss of generality,  we may assume $y_0 \in C^{2+\alpha}(\overline{\Omega})$ at $t = 0$ and construct a control function that vanishes in a short period of time at the beginning.

Let $\omega_2$ be a non-empty open set satisfying $\omega_2\Subset \omega$. Following the argument in Proposition \ref{pro1} and replacing $\omega$ in the proposition by $\omega_2$, we can show that there exist a control $\tilde{h}\in L^2(  \omega_2\times (0, T))$ and the associated solution $\tilde{y}$ satisfying $\tilde{y}(\cdot, T)=0$. Furthermore, it holds that
\begin{align}
\| \tilde{h}\|_{L^2(  \omega_2\times (0, T) )}
&\leq \exp( C(1+T^{-1}+T \|V\|_{L^\infty}+\vertiii{V} )) \|y_0\|_{L^2(\Omega)}.
\end{align}

With the aid of $\tilde{h}$ and $\tilde{y}$ at our disposal, we construct a control with the required regularity. Let $u$ be the solution of (\ref{linear-core-1}) with the control function $h=0$.  Due to the assumption $y_0\in C^{2+\alpha}(\bar\Omega)$, the Schauder estimates of heat equations imply that $u\in  C^{2+\alpha, 1+\frac{\alpha}{2}}(\bar Q_T)$; see e.g., \cite{L96}. Furthermore, the dependence of the constant on $T$ and $V$ is given precisely as follows:
\begin{align}
\|u\|_{C^{2+\alpha, 1+\frac{\alpha}{2}}(\bar Q_T)}\leq \exp({C(1 + T^{-1} + T \|V\|_{\infty} + \vertiii{V})}) \| y_0\|_{C^{2+\alpha}(\overline\Omega)}.
\label{uuu-1}
\end{align}
A proof of the above estimate is given in Lemma \ref{lemm-reg} in Appendix.

%we want to know how the constants in the Schauder estimates depend on $T$ and the norm of $V$, which are presented in the Lemma \ref{lemm-reg} in Appendix. That is,

We choose $\chi=\chi(t)\in C^\infty[0, T]$ to be a function such that $\chi=1$ in the neighborhood of $t=0$ and $\chi=0$ in the neighborhood of $t=T$. Let $\hat{y}=\tilde{y}-\chi(t) u$. Then $\hat{y}$ satisfies the following equation
\begin{equation*}
    \left\{
    \begin{aligned}
        \hat{y}_t-\Delta \hat{y}+V(x,t)\hat{y} &= -\chi'(t) u+\tilde{h} \mathds{1}_{\omega_2} \quad   &\text{in} \ & Q_T, \\
        \hat{y} &=0  \quad &\text{on}& \ \Sigma_T, \\
       \hat{y}(\cdot, 0) &=0    &\text{on}& \  \Omega,
    \end{aligned}
    \right.
\end{equation*}
and $\hat{y}(x, T)=0$ in $\Omega$.

Next, we choose $\omega_3$ to be another open set such that $\omega_2\Subset \omega_3 \Subset\omega$. Then we have
 \begin{equation}
    \left\{
    \begin{aligned}
        \hat{y}_t-\Delta \hat{y}+V(x,t)\hat{y} &= -\chi'(t) u  \quad   &\text{in}& \ \Omega\backslash\omega_2 \times (0, T], \\
        \hat{y} &=0  \quad &\text{on}& \ \Sigma_T, \\
       \hat{y}(\cdot, 0) &=0    &\text{on}& \  \Omega.
    \end{aligned}
    \right.
    \label{hat-yy}
\end{equation}
By the fact that $\hat{y}=\tilde{y}-\chi(t) u$ and $u$ satisfies (\ref{uuu-1}), from Lemma \ref{last-reg-lem}, we obtain
\begin{align}
\|\hat{y}\|_{C^{2+\alpha, 1+\frac{\alpha}{2}}(\OO\times (0, T))}\leq  \exp({C(1 + T^{-1} + T \|V\|_{\infty} + \vertiii{V})})\|y_0\|_{C^{2+\alpha}(\Omega)}
\label{sch-hat-y}
\end{align}
for some open set $\OO \supset\partial\omega_3 $.
Applying the Schauder estimates for $\hat{y}$ in (\ref{hat-yy}) in $\Omega\backslash\omega_3 \times (0, T])$, together with (\ref{uuu-1}) and
 (\ref{sch-hat-y}), we have
 \begin{align}
 \|\hat{y}\|_{C^{2+\alpha, 1+\frac{\alpha}{2}}(  \bar \Omega\backslash\omega_3 \times (0, T])}\leq \exp({C(1 + T^{-1} + T \|V\|_{\infty} + \vertiii{V})}) \|y_0\|_{C^{2+\alpha}(\Omega)}.
 \label{sch-hat-y-2}
 \end{align}

%Let us look that $\hat{y}$ on $\partial w_3\times (0, T)$. Since $\tilde{h}\in L^2(  \omega_2\times (0, T))$, then $\tilde{y}\in C^{2+\alpha, 1+\frac{\alpha}{2}}(\partial w_3\times (0, T])$. Starting from $y\in W^{2, 1, 2}$, we can iterate the interior local parabolic estimates ($L^p$ estimates and Schauder estimates) to lift the regularity.

Finally, we choose a cut-off function $\phi(x)\in C_0^\infty(\Omega)$ such that $0\leq \phi(x) \leq 1$ and $\phi=1$ in $\omega_4$ and $\phi = 0$ in $\Omega \setminus \omega$, where
 $\omega_3 \Subset \omega_4 \Subset \omega$.
  Then we select $y^\ast=(1-\phi) \hat{y}$. Taking into account the regularity of $\hat{y}$ and the assumption of $\phi$, we have $y^\ast\in C^{2+\alpha, 1+\frac{\alpha}{2}}(\bar Q_T)$. Let $y=y^\ast+\chi(t) u$. Note that $y(x, T)=0$ in $\Omega$. One has
\begin{equation}\label{eq.Ly=h}
    \left\{
    \begin{aligned}
        {y}_t-\Delta {y}+V(x,t){y} &= (1-\phi)\tilde{h} \mathds{1}_{\omega_2} -\phi \chi'(t) u+\Delta \phi \hat{y}+2 \nabla \phi\cdot \nabla \hat{y}=h\mathds{1}_{\omega}  \quad   &\text{in} \ & Q_T, \\
        {y} &=0  \quad &\text{on}& \ \Sigma_T, \\
       {y}(0, \cdot) &=y_0    &\text{on}& \  \Omega,
    \end{aligned}
    \right.
\end{equation}
where {$h(x,t)=-\phi \chi'(t) u+\Delta \phi \hat{y}+2 \nabla \phi\cdot \nabla \hat{y}$}
since $(1-\phi)\tilde{h} \mathds{1}_{\omega_2}=0$. By the regularity estimates of $\phi, \hat{y}$ and  $u$, we have
$h\in C^{\alpha, \frac{\alpha}{2}}(\bar \omega\times [0, T])$. Moreover, from the construction of $h(x, t)$, we know that
$h(x,t)$ is supported in  $\omega$ for every $t \in (0, T)$ and $h(x, 0)=0$ in $\Omega$.

The regularity estimates of $\hat{y}$ in (\ref{sch-hat-y-2}) and $u$ in (\ref{uuu-1}) further imply
\begin{align*}
\|h\|_{C^{\alpha, \frac{\alpha}{2}}( Q_T)  } = \|h\|_{C^{\alpha, \frac{\alpha}{2}}( \bar\omega \times [0, T])  } \leq  \exp({C(1 + T^{-1} + T \|V\|_{\infty} + \vertiii{V})}) \|y_0\|_{C^{2+\alpha}(\Omega)},
\end{align*}
and thus a global Schauder estimate for \eqref{eq.Ly=h}
\begin{align*}
\begin{aligned}
    \|y\|_{C^{2+\alpha, 1+\frac{\alpha}{2}}(\bar Q_T)  } & \leq \exp({C(1 + T^{-1} + T \|V\|_{\infty} + \vertiii{V})}) ( \| y_0 \|_{C^{2+\alpha}(\Omega)} + \| h \|_{C^{\alpha, \frac{\alpha}{2}}( Q_T)} ) \\
    & \leq  \exp({C(1 + T^{-1} + T \|V\|_{\infty} + \vertiii{V})}) \|y_0\|_{C^{2+\alpha}(\Omega)}.
\end{aligned}
\end{align*}
Therefore, we have completed the proof of Theorem \ref{th2}.
\end{proof}

\section{One-dimensional heat equation}
In this section, we show the optimal  observability inequalities for the one-dimensional heat equation. The observation region is concerned with a subset of positive measure in $\Omega$ and a subset of positive measure in $(0, T)$. This requires a well-known technique involving
the following lemma, which is based on the property of density points for sets of positive measure on $(0, T)$; see e.g., \cite{PW13}.
\begin{lemma}
Let $E$ be a measurable subset of positive measure in $(0, T)$ and $k$ be a density point of $E$. Then for any $\gamma>1$, there exists $k_1\in (k, T)$ such that the sequence defined by
\begin{align*}
k_{m+1}-k= \gamma^{-m}(k_1-k)
\end{align*}
satisfies
\begin{align*}
|E\cap (k_{m+1}, k_m)|\geq \frac{ (k_m-k_{m+1})}{3}.
\end{align*}
\label{densi}
\end{lemma}

We also need the following results for the propagation of smallness results for holomorphic functions, whose proof relies on a Remez-type inequality; see \cite{Ma04} or \cite[Proposition 2.3]{Zhu23}.  
%The results seem to be elementary. However, we can not find out the literature. Following the proof  of Proposition 2.1 in \cite{Zhu23}, which is on the propagation of smallness for solutions of second order elliptic equations in the divergence form, we can  readily provide the arguments in the following lemma by carefully investigating the proof in \cite{Zhu23}. Furthermore, the  explicit dependence of $\alpha$ can be shown below.  
Let $\mathcal{E}$ be a measurable set on the line segment $(-\frac{1}{2}, \frac{1}{2})$ in the two dimensional ball $B_{1/2}$ with $|\mathcal{E}|>0$, where $|\mathcal{E}|$ is the Lebesgue measure on the line.
 \begin{lemma}
 Let $h(z)$ be a holomorphic function in $B_4\subset \mathbb C$. Then there exists a constant $0<\alpha<1$ depending on $|\mathcal{E}|$ such that
\begin{align}
  \sup_{B_1}|h|\leq  \sup_{\mathcal{E}}|h|^{\alpha} \sup_{B_4}|h|^{1-\alpha},
  \label{three-hol}
\end{align}
where $\alpha=\frac{1}{1+C+C\log\frac{C}{|\mathcal{E}|}}$ and $C>1$ is a universal constant.
 \end{lemma}
%\jz{This lemma has been shown in \cite{Zhu23}, right?}  \textcolor{blue}{Basically I used the method in this paper. The results seems to be known. We had better write the lemma for a holomorphic function, because we apply it for the  holomorhphic functions later on and quantitatively characterize the norm of $V$} \jz{I agree that this should be a known result. It's better to find a more classical reference. I think (2.9) in \cite{Zhu23} is exactly the above inequality. }

For our application, we would like to replace the $\sup_{\mathcal{E}}|h|$ in the inequality (\ref{three-hol}) by the norm $\|h\|_{L^2(\mathcal{E})}$. To this end, we set 
\begin{align*}
    \mathcal{E}'_1=\{x\in \mathcal{E}| h(x)\geq \frac{4}{ |\mathcal{E}|^{1/2} }\|h\|_{L^2(\mathcal{E})}\}.
\end{align*}
We claim that $|\mathcal{E}'_1|\leq \frac{1}{2}|\mathcal{E}|$. Otherwise, if $|\mathcal{E}'_1|> \frac{1}{2}|\mathcal{E}|$, then
\begin{align*}
   \|h\|_{L^2(\mathcal{E}'_1)} \geq \frac{4}{ {|\mathcal{E}|}^{1/2}} \|h\|_{L^2(\mathcal{E})} (\frac12 |\mathcal{E}| )^{1/2}\geq 2\|h\|_{L^2(\mathcal{E})}.
\end{align*}
Obviously, it is a contradiction. This shows the claim. We introduce 
\begin{align*}
   \mathcal{E}_1^0=\mathcal{E}\backslash \mathcal{E}'_1=\{x\in \mathcal{E}| h(x) < \frac{4}{|\mathcal{E}|^{1/2}} \|h\|_{L^2(\mathcal{E})}\}.
\end{align*}
Then $|\mathcal{E}_1^0|\geq \frac{1}{2} |\mathcal{E}|$. Applying the inequality (\ref{three-hol}) with $\mathcal{E}$ replaced by $\mathcal{E}_1^0$, we get
\begin{align}
  \sup_{B_1}|h|&\leq  (\frac{4}{|\mathcal{E}|^{1/2}})^\alpha \|h\|_{L^2(\mathcal{E})}^{\alpha} \sup_{B_{4}}|h |^{1-\alpha} \nonumber\\
  &\leq C\|h\|_{L^2(\mathcal{E})}^{\alpha}\sup_{B_{4}}|h|^{1-\alpha}.
  \label{three-measure-1}
\end{align}

To show the observability inequalities,  we will first obtain some spectral inequalities for a linear combination of eigenfunctions,  which might be of independent interest. We consider the eigenvalue problem for the one-dimensional Schr\"odinger operator (or Hill operator) 
\begin{align}
\left\{
\begin{aligned}
     -\partial_{xx}\phi_k +V(x)\phi_k & =\lambda_k \phi_k  \quad & \mbox{in} \ \Omega, \nonumber\\
 \phi_k(x) & =0 \quad & \mbox{on} \ \partial\Omega,
\end{aligned}
\right.
\end{align}
where $\Omega=(0, \frac{1}{2})$, the eigenvalues $\{\lambda_k: k=1,2,\cdots \}$ are ordered nondecreasingly, counted with multiplicity. Note that the smallest eigenvalue $\lambda_1$ is larger than $\inf_{\Omega} V$. Moreover, the eigenfunctions $\{ \phi_k: k =1,2,\cdots \}$ are orthonormal in $L^2(\Omega)$. Let $-\tilde{\Delta}=-\partial_{xx}+V(x)$ and $P_\lambda=P_\lambda(-\tilde{\Delta}): L^2(\Omega) \to L^2(\Omega)$ be the projection onto the eigenspace generated by $\{\phi_k: \lambda_k\leq \lambda\}$. Without loss of generality, we assume that $\|V\|_{\infty}= M$ for some large constant $M>3$. 
%The sign of $V(x)$ plays the role of obtaining the sharp versions of spectral inequalities based on the current literature. 
By some ideas in \cite{KSW15}, we can show the following sharp spectral inequalities for the one-dimensional Schr\"odinger equation. 

%In case of nonnegative potentials, the spectral inequality is sharp; while there is a logarithmic loss in the exponent for general potentials that may change signs. The later case will not be used for our main results, while we keep it here for possible independent interest.

\begin{lemma}
Let $\omega$ be a measurable subset of positive measure on $\Omega$. There exists a constant $C$ depending on $\omega$ such that 
\begin{align}
\| \phi\|_{L^2(\Omega)}\leq  e^{C( M^{1/2}+{\lambda}^{1/2})} \| \phi\|_{L^2(\omega)} \quad \text{for all } \phi\in Ran(P_\lambda(-\tilde{\Delta})).
\label{spec-in-3}
\end{align}
%Case 2: If $V(x)$ changes signs  in $\Omega$, then
%   \begin{align}
%\| \phi\|_{L^2(\Omega)}\leq  e^{C( M^{1/2}(\log M)^{3/2}+(\log M)  |\lambda|^{1/2})} \| \phi\|_{L^2(\omega)} \quad \text{for all } \phi\in Ran(P_\lambda(-\tilde{\Delta})).
%\label{spec-in-3-1}
%\end{align}
\end{lemma}

\begin{proof}
 We first deal with the nonnegative potential $V(x)$.
As $\Omega=(0,\frac{1}{2})$ and $\phi_k(0) = 0$, we extend $\phi_k$ by an odd reflection to be in $(-\frac{1}{2}, \frac{1}{2})$ across $x=0$. Then we do an even extension for $V(x)$ across $x=0$ in $(-\frac{1}{2}, \frac{1}{2})$. We still write the extended functions as $\phi_k$ and $V(x)$. Hence we have
\begin{align}
\left\{
\begin{aligned}
     -\partial_{xx}\phi_k +V(x)\phi_k & =\lambda_k \phi_k  \quad  & \mbox{in} \ \ {\Omega}_1:=(-\frac{1}{2}, \frac{1}{2}), \nonumber\\
 \phi_k(x) & =0 \quad & \mbox{on} \ \partial{\Omega}_1.
\end{aligned}
\right.
\end{align}
We can further extend $\phi_k$ as a periodic function in $\mathbb R$ with period $1$. In particular, $\phi_k$ satisfies the above equation in $\tilde{\Omega}:=(-1, 1)$ and vanishes on $\partial \tilde{\Omega}$. Note that $\phi_k\in C^{1, 1}(\tilde{\Omega})$.

Since $V(x)\geq 0$, then $\lambda_1>\inf_\Omega V \geq 0$. For any $\phi \in Ran(P_\lambda(-\tilde{\Delta}))$, we can write $\phi=\sum_{0<\lambda_k\leq \lambda} \alpha_k \phi_k$ with $\alpha_k=(\phi_k, \phi)$. We construct 
\begin{align} u(x, y)=\sum_{0<\lambda_k\leq \lambda} \alpha_k \cosh(\sqrt{\lambda_k} y)\phi_k(x).
\label{uuu1}\end{align}
Note that $\frac{\partial u}{\partial y}=0$ on $\{(x,y) \in \tilde{\Omega}\times \mathbb R| y=0\}$. It is easy to see that
\begin{align}
-\Delta u+V(x)u=0 \quad \mbox{in} \ \tilde{\Omega}\times (-1, 1),
\label{www-3}
\end{align}
where we used the usual notation $\Delta =\partial_{xx} +\partial_{yy}$.

Next, we construct a multiplier to convert \eqref{www-3} into a second order elliptic equation without lower order terms.  We  first prove the existence of a positive solution of the following equation. 
\begin{align}
\left\{
\begin{aligned}
     -\partial_{xx}w +V(x) w & =0  \quad & \mbox{in} \ \tilde{\Omega}, \\
 w & = e^{\sqrt{M}} \quad & \mbox{on} \ \partial\tilde{\Omega}.
\end{aligned}
\right.
\label{build}
\end{align}
Choose $\hat{w}_1= e^{\sqrt{M}x}$. We can check that $-\partial_{xx}\hat{w}_1 +V(x) \hat{w}_1 \leq 0$. Thus, $\hat{w}_1$ is a subsolution of (\ref{build}). Let $\hat{w}_2= e^{\sqrt{M}}$. Then $-\partial_{xx}\hat{w}_2 +V(x) \hat{w}_2 \geq 0$ as $V(x) \geq 0$. Therefore, $\hat{w}_2$ is a supersolution of (\ref{build}). By the sub-supersolution method, there exists a positive solution $w$ of (\ref{build}) such that $\hat{w}_1\leq w\leq \hat{w}_2$. Therefore, we have 
\begin{align} e^{-\sqrt{M} }\leq w(x)\leq e^{\sqrt{M} }.
\label{www-1}\end{align} 
Using \eqref{www-1} and the equation \eqref{build}, we can get the estimate of $\partial_{xx} w$ and hence
\begin{align} \| w\|_{C^{1,1}(\tilde{\Omega})} \leq e^{C\sqrt{M} }.
\label{www-2}\end{align}

Now, observe that $w = w(x)$ given above can also be viewed as a positive solution of 
\begin{align*}
 -\Delta w+V(x) w=0    \quad \mbox{in} \ \tilde{\Omega}\times (-1, 1).
\end{align*}
We consider $v(x,y)=\frac{u(x, y)}{w(x)}$. Then $v(x,y)$ satisfies 
\begin{align}
    {\rm div}(w^2 \nabla v)=0 \quad \mbox{in} \ \tilde{\Omega}\times (-1, 1).
\end{align}
We will use some techniques from complex analysis, which have been applied in \cite{KSW15} in the study of Landis' conjecture in the plane. We construct a stream function $\tilde{v}$ associated with $v$ in $B_1$. That is, $\tilde{v}$ satisfies 
\begin{align}
\left\{
\begin{aligned}
     \partial_y \tilde{v}  & = w^2 \partial_x {v}, \\
-\partial_x \tilde{v}  & = w^2 \partial_y {v}.
\end{aligned}
\right.
\label{anti-sy}
\end{align}
We can choose $\tilde{v}(0)=0$. Let $g=w^2 v+i \tilde{v}.$ Then 
\begin{align}
    \overline{\partial} g=\frac{\overline{\partial} w^2}{2w^2}(g+\bar g)    \quad \mbox{in} \ B_1, \label{complex-1}
\end{align}
where $\overline{\partial}=\frac{1}{2}({\partial x}+ i{\partial y})$.
Hence we can write it as
\begin{align*}
    \overline{\partial} g= q_0 g+ q_0 \bar g,
\end{align*}
where $q_0=\frac{\overline{\partial} w^2}{2w^2}= \overline{\partial} \log (w)$. Furthermore,
we can reduce it to
\begin{align}
   \overline{\partial} g=\tilde{q}_0 g   \quad \mbox{in}  \ B_1,  \label{complex-2}
\end{align}
where
\begin{align*}
\tilde{q}_0= \left\{ 
\begin{array}{lll}
{q}_0+ \frac{{q}_0 \bar g}{g}, \quad & \mbox{if} \ g\not=0, \nonumber \\
0,  \quad & \mbox{otherwise}.
\end{array}
\right.
\end{align*}
It was shown in \cite[Lemma 2.2]{KSW15} that $\|\tilde{q}_0\|_{L^\infty(B_1)}\leq 2\| q_0\|_{L^\infty(B_1)} \leq C\sqrt{M}$.
Furthermore, it can be shown that, see e.g. Theorem 4.1 in \cite{Bo09},
\begin{align}g(z)=e^{W(z)}h(z) \quad \mbox{in} \ B_{4/5},
\label{valid-1}
\end{align}
where $h$ is a holomorphic function in $B_{4/5}$,
\begin{align*}
    W(z)=-\frac{1}{\pi}\int_{B_{4/5}}\frac{\tilde{q}_0}{\xi-z} d\xi,
\end{align*}
and  $z=x+iy$. Moreover, the bound of $\tilde{q}_0$ leads to $\|W\|_{L^\infty (B_{4/5})}\leq C\sqrt{M} $.

Next we apply the propagation of smallness in (\ref{three-measure-1}). Let $\omega_1$ be a measurable subset on the line $B_{r/2}\cap\{(x,y)|y=0\}$ with  $|\omega_1|>0$.  We identify the ball in the complex plane with the ball in $\mathbb R^2$. Replacing $\mathcal{E}$ by $\omega_1$ and letting $h=e^{-W(z)}g(z)$ in \eqref{three-measure-1}, we have
\begin{align}
  \sup_{B_r}|e^{-W(z)}g(z)|\leq C \|e^{-W(z)}g(z)\|^{\alpha}_{L^2(\omega_1)} \sup_{B_{4r}}|e^{-W(z)}g(z)|^{1-\alpha}   
  \label{hada-1}
\end{align}
for some $r<r_0\leq \frac{1}{10}$, where $r_0$ is some universal constant.
Thanks to the expression of $g$ and the conditions for $W(z)$, we get
\begin{align}
  \sup_{B_r}|v|\leq e^{C\sqrt{M}} \|w^2 v+i\tilde{v}\|^{\alpha}_{L^2(\omega_1)} \sup_{B_{4r}}|w^2 v+i\tilde{v} |^{1-\alpha}.   
  \label{hada-2}
\end{align}
Since $\frac{\partial u}{\partial y}=0$ on $\{(x,y)\in B_1|y=0\}$, then $\frac{\partial v}{\partial y}=0$ {on $\{(x,y)\in B_1|y=0\}$} as $w(x)$ depends only on $x$. From (\ref{anti-sy}), it follows that  $\frac{\partial \tilde{v}}{\partial x}=0$ on $\{(x,y)\in B_1|y=0\}$. As $\tilde{v}(0)=0$, then $\tilde{v}=0$ on $\{(x,y)\in B_1|y=0\}$.
It holds that $\nabla v=\nabla (\frac{u}{w})=\frac{\nabla u w-\nabla w u}{w^2}.$
We can show that 
\begin{align}
  \sup_{B_{4r}}|\tilde{v} |   &\leq C r e^{C\sqrt{M}} \sup_{B_{4r}} |\nabla v| \nonumber \\
  & \leq e^{C\sqrt{M}} \sup_{B_{5r}} | u|,
  \label{hada-3}
\end{align}
where we used the estimates (\ref{www-1}),  (\ref{www-2}), \eqref{anti-sy}, and the elliptic estimates for $u$ in (\ref{www-3}).
Therefore, we obtain that
\begin{align}
  \sup_{B_r}|u|\leq e^{C\sqrt{M}}  \|u\|^{\alpha}_{L^2(\omega_1)} \sup_{B_{5r}}|u |^{1-\alpha}.    
  \label{three-measure}
\end{align}

Recall that $u(x,y)$ is given in (\ref{uuu1}).
The orthogonality of ${\phi}_k$ in $L^2(\widetilde{\Omega})$ leads to
\begin{equation}\label{est.orthogonal}
    \int_{\widetilde{\Omega} \times (-1,1)} |u(x,y)|^2 dxdy \le \exp( C(1+ \lambda^{1/2} )) \int_{\widetilde{\Omega}} |u(x,0)|^2 dx,
\end{equation}
% Then $u(x,l)$ satisfies
% \begin{align}
% - k^{-\frac{1}{2}} \partial_x (k^{-\frac{1}{2}} \partial_x  u)-\partial_{ll} u+ (\|V\|_\infty+ V) u=0.
% \end{align}
% We can write it in term of  Laplace-beltrami operator
% \begin{align*}
% -\Delta_g u+ (\|V\|_\infty+ V) u=0,
% \end{align*}
% where $g=\begin{pmatrix} k, \  0 \\
%  0,  \ 1
%  \end{pmatrix}$ and $\Delta_g =\frac{1}{\sqrt{|g|}} {\rm{div}} (g^{-1} \sqrt{|g|}\nabla)$. Since $g$ is Lipschitz continuous, we can further apply a local isothermal coordinates to reduce to Laplace operator.
%   \begin{align*}
% -\Delta u+ (\|V\|_\infty+ V)\psi u=0
% \end{align*}
% where $\psi>0$ is Lipschitz continuous depending on $g$.
 Let $\bar{x} \in {\Omega}$ be such that
\begin{equation*}
    M_0 := |u(\bar x, 0)|=\sup_{\frac{1}{2}\widetilde{\Omega}\times \{0\}} |u|=\sup_{{\Omega}} |\phi|,
\end{equation*}
and $\omega$ be the given measurable
subset of positive measure in $\Omega$.
%We choose $\omega_1=\{x\in \omega| \dist(x, \partial\omega)\geq \frac{d}{2}\}$.
Then there exists a point $x_\ast \in \omega$ such that $|B_r(x_\ast, 0)\cap \omega|>c|\omega|$ for $\frac12 r_0 \leq r\leq r_0$ and some small $c$. Let $\omega_1=B_r(x_\ast, 0)\cap \omega$. 
%\begin{equation*}
%    M_\ast:= |u(x_\ast, 0)| = \sup_{{\omega_1} \times \{0\}} |u|.
%\end{equation*}
We apply (\ref{three-measure}) at $(x_\ast, 0)$ with  $\frac12 r_0 \leq r\leq r_0$,
\begin{equation}\label{Linfty.xstar}
\begin{aligned}
   \|u\|_{L^\infty(B_{2r}(x_\ast, 0))} & \leq  \exp(C M^{1/2}) \|\phi\|^{\alpha}_{L^2(\omega_1)} \|u\|^{1-\alpha}_{L^\infty(B_{10r}(x_\ast, 0))} \\
   & \le \exp(C M^{1/2})\|\phi\|^{\alpha}_{L^2(\omega)} \|u\|^{1-\alpha}_{L^2(\widetilde{\Omega} \times (-1,1))} \\
   & \le \exp( C(\lambda^{1/2} + M^{1/2} )) \|\phi\|^{\alpha}_{L^2(\omega)} \|u(\cdot,0)\|^{1-\alpha}_{L^2(\widetilde{\Omega} )} \\
   & \le \exp( C(\lambda^{1/2} +M^{1/2} )) \|\phi\|^{\alpha}_{L^2(\omega)} M_0^{1-\alpha},
\end{aligned}
\end{equation}
where we have used the $L^\infty$ regularity for the elliptic equation \eqref{www-3} in the second inequality, and \eqref{est.orthogonal} in the third inequality.

Now if we can show
\begin{equation}\label{est.M0}
    M_0 \le \exp(C(M^{1/2} + \lambda^{1/2})) \|u\|_{L^\infty(B_{2r}(x_\ast, 0))},
\end{equation}
then \eqref{Linfty.xstar} implies
\begin{equation}
    M_0 \le \exp( C(\lambda^{1/2} +M^{1/2} )) \|\phi\|_{L^2(\omega)},
\end{equation}
which leads to \eqref{spec-in-3}. Thus, it suffices to show \eqref{est.M0}. Indeed, this follows from the three-ball inequality and a standard argument.

% It holds that $\|u(x, 0)\|_{L^2(\tilde{\Omega})}\leq CM_0$. Then  $\|u(x, \theta)\|_{L^2(\tilde{\Omega}\times (-2, 2))}\leq e^{C (\|V\|^{\frac{1}{2}}_\infty+\lambda^{\frac{1}{2}})} M_0$. Thus, by elliptic estimates, it also holds that $\|u(x, \theta)\|_{L^\infty(\tilde{\Omega}\times (-1, 1))}\leq e^{C \|V\|_{\infty}^{\frac{1}{2}}}\|u(x, \theta)\|_{L^2(\tilde{\Omega}\times (-2, 2))}\leq e^{C (\|V\|^{\frac{1}{2}}_\infty+\lambda^{\frac{1}{2}})} M_0$. \\
% Note  we can choose $\|V\|^{\frac{1}{2}}_\infty\geq 1$. Otherwise, we adopt the constants $e^{C (1+\|V\|^{\frac{1}{2}}_\infty+\lambda^{\frac{1}{2}})}$ in these inequalities .
% }

In fact, as $ V\geq 0$, by the proof of \cite[Theorem 2.5]{KSW15}, we have the following three-ball inequality: for any $z = (x,y) \in \widetilde{\Omega} \times (-\frac23,\frac23)$ and $r<r_0$, it holds that
\begin{align}
\|u\|_{L^\infty(B_{2r}(z))}\leq  \exp{(C M^{1/2})} \|u\|^{1-\alpha_1}_{L^\infty(B_{r}(z))}\|u\|^{\alpha_1}_{L^\infty(B_{3r}(z))},
\label{three-kenig}
\end{align}
where $0<\alpha_1<1$ and $r_0>0$ are some universal constants. Then for any $z_1$ with $|z_1 - z|\leq r$, we have $B_r(z_1) \subset B_{2r}(z)$ and hence by \eqref{three-kenig}
\begin{equation}
    \|u\|_{L^\infty(B_{r}(z_1))} \le \exp(CM^{1/2}) \|u\|^{1-\alpha_1}_{L^\infty(B_{r}(z))} M_0^{\alpha_1}.
\end{equation}

% Next we apply the three-ball inequality (\ref{three-kenig}).
% Let 
% \begin{align} M_\ast =\exp( C(1+ \lambda^{1/2} + M^{1/2} )) \|\phi\|^{\alpha}_{L^2(\omega)} M_0^{1-\alpha}. 
% \label{denote-1}
% \end{align}
%  We choose a point $x_1\in \tilde{\Omega}$ such that $B_r(x_1, 0)\subset B_{2r}(x_\ast, 0)$. We obtain that
% \begin{align*}
% \|u\|_{L^\infty(B_{r}(x_1, 0))}\leq e^{C (M^{1/2}+\lambda^{1/2})}  M_\ast^{1-\alpha_1} M^{\alpha_1}_0.  
% \end{align*}
Let $x_0=x_\ast, x_1, \cdots,  x_{m-1}, x_m=\bar x$ be a finite number of points such that $|x_{i+1} - x_i| < r \in (\frac12 r_0, r_0)$, where the number $m$ depends only on $\tilde{\Omega}$ and $r_0$. At each point $(x_i, 0)$, we apply the three-ball inequality (\ref{three-kenig}) repeatedly with $z = (x_i,0)$ and $z_1 = (x_{i+1},0)$ and obtain 
% and use the fact that
% $\|u\|_{L^\infty(B_{r}(x_{i+1}, 0))} \leq  \|u\|_{L^\infty (B_{2r}(x_{i}, 0))}$. We will get
\begin{align}
\|u\|_{L^\infty(B_{r}(\bar{x},0))} \leq \exp(C_m (M^{1/2}+\lambda^{1/2})) \|u\|_{L^\infty(B_{r}(x_*,0))}^{(1-\alpha_1)^m} M_0^{1-(1-\alpha_1)^m}.
\label{simple-1}
\end{align}
By definition, the left-hand side of \eqref{simple-1} is equal to $M_0$. Hence, simplifying the last inequality, we obtain \eqref{est.M0}. The proof of (\ref{spec-in-3} ) is complete for nonneagive potential $V(x)$.

% and applying (\ref{denote-1}) to have 
% \begin{align*}
% M_0\leq e^{C_m (M^{1/2}+\lambda^{1/2})} \|\phi\|^{\alpha}_{L^2(\omega)} M_0^{1-\alpha}.
% \end{align*}
%\jz{This $\alpha$ is different from the previous one in \eqref{three-measure-1}, right? (\textcolor{blue}{The same, this $\alpha$ comes from $M_\ast$. simply the inequality (\ref{simple-1}) to get it. } The propagation argument based on $\eqref{three-kenig}$ is standard and can be omitted.} \textcolor{blue} {Let us Keep it. This three-ball inequality $\eqref{three-kenig}$ is not well-known. and we have extended the region. }
% That is,
% \begin{align}
% \sup_{\Omega\times \{0\}} |u(x,0)|\leq e^{C (M^{1/2}+\lambda^{1/2})}\|\phi\|_{L^2(\omega)}.
% \label{spec-u}
% \end{align}
% As $u(x,0)=\phi(x)$, the estimate (\ref{spec-u}) implies  that
% \begin{align}
% \| \phi\|_{L^2(\Omega)}\leq  e^{C({M^{1/2}}+{\lambda}^{1/2})} \| \phi\|_{L^2(\omega)}.
% \label{spec-in}
% \end{align}
% Hence, we obtained (\ref{spec-in-3}).
%\jz{The last inequality is the main estimate to be used. The argument below seems to be a standard proof. Please refer to the source of this proof.}
%\textcolor{blue}{We need to include it because we have the dependence of $V$ here. Some literature says as as follows}

 Now we deal with the general case $V(x)$ that changes signs. The eigenvalues $\lambda_k$ are discrete and approach infinity as $k
\to \infty$. Notice that there may be finitely many of negative eigenvalues and the smallest eigenvalue  $\lambda_1$ is larger than  $-M$.  
We choose $\tilde{V}(x)=V(x)+M$. Then 
\begin{align}
 -\partial^2_x \phi_k+\tilde{V}(x) \phi_k=\tilde{\lambda}_k\phi_k,  
\end{align}
where $\tilde{\lambda}_k=\lambda_k+M\geq 0$, $\tilde{V}(x)\geq 0$ and $\|\tilde{V}\|_{L^\infty}\leq 2M$.
Let $\phi=\sum_{-M\leq \lambda_k\leq \lambda} \alpha_k \phi_k$ with $\alpha_k=(\phi_k, \phi)$.   We construct 
\begin{align} u(x, y)=\sum_{-M<\lambda_k\leq \lambda} \alpha_k \cosh(\sqrt{\tilde{\lambda}_k }y)\phi_k(x).
\label{uuu1}\end{align}
We still have $\frac{\partial u}{\partial y}=0$ on $\{(x,y) \in \tilde{\Omega}\times \mathbb R| y=0\}$, and 
\begin{align}
-\Delta u+\tilde{V}(x)u=0 \quad \mbox{in} \ \tilde{\Omega}\times (-1, 1).
\label{www-3-6}
\end{align}
Thus, we can apply the results for the nonnegative potential in the previous calculations to (\ref{www-3-6}). We have 
\begin{align}
\| \phi\|_{L^2(\Omega)}\leq  e^{C( (2M)^{1/2}+{(\lambda+M)}^{1/2})} \| \phi\|_{L^2(\omega)}.
\end{align}
Hence the spectral inequalities (\ref{spec-in-3}) is arrived for sign changing potential $V(x)$. Therefore, the lemma is completed.

\end{proof}

Next we deduce the observability inequalities in Theorem \ref{th3} from the spectral inequalities (\ref{spec-in-3}) following the proof in \cite{AEWZ14}, which in turn relies on the ideas in \cite{M10}, {interpolation inequalities} (see, e.g., (\ref{est.et2Delta}) below)
and telescoping series method in \cite{PW13}. Moreover, we explicitly show the dependence of the norm $\|V\|_\infty$ in the arguments.

\begin{proof}[Proof of Theorem \ref{th3}] 

We first consider the case for nonnegative potential $V \ge 0$. Let $P^\lambda=I-P_\lambda$ and $\|V\|_{\infty}=M$. We introduce the semigroup $e^{t\tilde{\Delta}}$ generated by $\tilde{\Delta} = \Delta - V$, which can be defined as $e^{t\tilde{\Delta}} f=\sum a_i e^{-\lambda_i t} \phi_i(x)$ for any $f\in L^2(\Omega)$, where $a_i=(f,\phi_i)$. Note that
\begin{equation}\label{est.semigroup}
    \| e^{t\tilde{\Delta}} f \|_{L^2(\Omega)}^2 = \sum_{i} |a_i|^2 e^{-2\lambda_i t}.
\end{equation}

Let $t>s\geq 0$ and $\lambda>0$. Thanks to (\ref{spec-in-3}),
we have
\begin{align}
\| e^{t\tilde{\Delta}} f\|_{L^2(\Omega)}&\leq \| e^{t\tilde{\Delta}} P_\lambda f\|_{L^2(\Omega)}+ \| e^{t\tilde{\Delta}} P^\lambda f\|_{L^2(\Omega)} \nonumber \\
&\leq C e^{C({M^{1/2}}+{\lambda}^{1/2})}  \| e^{t\tilde{\Delta}} P_\lambda  f \|_{L^2(\omega)} + \| e^{t\tilde{\Delta}} P^\lambda f\|_{L^2(\Omega)}\nonumber \\
&\leq Ce^{C(M^{1/2}+{\lambda}^{1/2})} ( \| e^{t\tilde{\Delta}}  f\|_{L^2(\omega)}+ \|e^{t\tilde{\Delta}} P^\lambda  f\|_{L^2(\Omega)} ) + \| e^{t\tilde{\Delta}} P^\lambda f \|_{L^2(\Omega)}  \nonumber \\
&\leq C_1 e^{C(M^{1/2}+{\lambda}^{1/2})} ( \| e^{t\tilde{\Delta}}  f\|_{L^2(\omega)}+ \|e^{t\tilde{\Delta}} P^\lambda  f\|_{L^2(\Omega)} )\nonumber \\
&\leq C_1 e^{C(M^{1/2}+{\lambda}^{1/2})} ( \| e^{t\tilde{\Delta}}  f\|_{L^2(\omega)}+ e^{-\lambda(t-s)} \|e^{s\tilde{\Delta}} P^\lambda  f\|_{L^2(\Omega)} ),
\label{how-long}
\end{align}
where we used \eqref{est.semigroup} in the last inequality.
Let $\delta \in (0,1)$, to be determined later. It is obvious that
\begin{align}
\sup_{\lambda \geq 0}  e^{C\sqrt{\lambda}- \delta \lambda(t-s)} = e^{ \frac{C^2}{4\delta(t-s)}}=e^{ \frac{K}{\delta(t-s)}}
\end{align}
for some fixed constant $K = C^2/4>0$. 
Then we have
\begin{align*}
\| e^{t\tilde{\Delta}} f\|_{L^2(\Omega)}&\leq C_1 e^{CM^{1/2}}  e^{ \frac{K}{\delta(t-s)}} (  e^{\delta \lambda (t-s)} \| e^{t\tilde{\Delta}}  f\|_{L^2(\omega)}+ e^{-(1-\delta)\lambda(t-s)} \|e^{s\tilde{\Delta}}   f\|_{L^2(\Omega)}).
\end{align*}
We minimize the right-hand side of the last inequality.
Since
\begin{align*}
\| e^{t\tilde{\Delta}} f\|_{L^2(\omega)}\leq \| e^{t\tilde{\Delta}} f\|_{L^2(\Omega)}\leq e^{-\lambda_1(t-s)} \| e^{s\tilde{\Delta}} f\|_{L^2(\Omega)}\leq  \| e^{s\tilde{\Delta}} f\|_{L^2(\Omega)}
\end{align*}
as $\lambda_1\geq 0$ and $t>s$,
we can choose  $\lambda$, which  is some positive free parameter, such that
\begin{align}
 e^{\lambda(t-s)}=\frac{ \| e^{s\tilde{\Delta}} f\|_{L^2(\Omega)}} {  \| e^{t\tilde{\Delta}} f\|_{L^2(\omega)}}.
\end{align}
Then we have
\begin{align}
\| e^{t\tilde{\Delta}} f\|_{L^2(\Omega)}\leq 2 C_1  e^{CM^{1/2}}  e^{ \frac{K}{\delta(t-s)}} \| e^{t\tilde{\Delta}} f\|^{1-\delta}_{L^2(\omega)}  \| e^{s\tilde{\Delta}} f\|^{\delta}_{L^2(\Omega)}.
\end{align}
We will consider $t\in (t_1, t_2)$ for any $t_1<t_2$. Choosing $s=t_1$, we get from \eqref{reg-first} with $V_- = 0$ that
\begin{align}
\| e^{t_2\tilde{\Delta}} f\|_{L^2(\Omega)}&\leq  \| e^{t\tilde{\Delta}} f\|_{L^2(\Omega)} \nonumber \\
&\leq 2 C_1  e^{C M^{1/2}}  e^{ \frac{K}{\delta(t-t_1)}} \| e^{t\tilde{\Delta}} f\|^{1-\delta}_{L^2(\omega)}  \| e^{t_1\tilde{\Delta}} f\|^{\delta}_{L^2(\Omega)}.\label{est.et2Delta}
\end{align}
Let $\gamma=2$, $t_1=k_{m+1}$ and $t_2=k_m$. From Lemma \ref{densi}, we get that  $|E\cap  (t_1,  t_2)|\geq \frac{t_2-t_1}{3}$. Then it holds that
\begin{align}
 |E\cap  (t_1+\frac{t_2-t_1}{4}, t_2)|\geq \frac{t_2-t_1}{12}.
 \end{align}
Integrating \eqref{est.et2Delta} over $t\in E\cap (t_1+\frac{t_2-t_1}{4}, t_2)$ and using the H\"older's inequality yield
\begin{align}
& \| e^{t_2\tilde{\Delta}} f\|_{L^2(\Omega)} \nonumber \\
& \leq 2 C_1  e^{C M^{1/2}}  e^{ \frac{4K}{\delta(t_2-t_1)}}
\bigg\{ \int^{t_2}_{t_1+\frac{t_2-t_1}{4}}\mathds{1}_E(t) \| e^{t\tilde{\Delta}} f\|_{L^2(\omega)}\, dt \bigg\}^{{1-\delta}}  \| e^{t_1\tilde{\Delta}} f\|^{\delta}_{L^2(\Omega)}.
\end{align}
Let $a$ be some positive constant to be determined. By the Young's inequality $AB\leq (1-\delta)A^{\frac{1}{1-\delta}}+\delta B^{\frac{1}{\delta}}$, we get
\begin{align}
& e^{-\frac{a}{t_2-t_1}} \| e^{t_2\tilde{\Delta}} f\|_{L^2(\Omega)} \nonumber \\
&\leq C_1  e^{C M^{1/2}} 
\bigg\{ \int^{t_2}_{t_1+\frac{t_2-t_1}{4}} \mathds{1}_E(t)\| e^{t\tilde{\Delta}} f\|_{L^2(\omega)}\, dt \bigg\}^{{1-\delta}}   e^{ \frac{4K/\delta-a}{(t_2-t_1)}} \| e^{t_1\tilde{\Delta}} f\|^{\delta}_{L^2(\Omega)}\nonumber \\
&\leq  C_1 (1-\delta) e^{{\frac{1}{1-\delta}}CM^{1/2}} 
\bigg\{ \int^{t_2}_{t_1+\frac{t_2-t_1}{4}}\mathds{1}_E(t) \| e^{t\tilde{\Delta}} f\|_{L^2(\omega)}\, dt \bigg\} \nonumber \\ 
 & \qquad  + \delta e^{ \frac{4K/\delta-a}{\delta(t_2-t_1)}}  \| e^{t_1\tilde{\Delta}} f\|_{L^2(\Omega)}.
\end{align}
Now we choose the constants $a$ and $\delta$ such that
\begin{align}
 \frac{a- 4K/\delta}{\delta}>2a.
 \label{aaa-1}
\end{align}
%\jz{should be $\frac{a}{2(1-\delta)} \geq d$?}
Hence, we have
\begin{align}
e^{-\frac{a}{t_2-t_1}} \| e^{t_2\tilde{\Delta}} f\|_{L^2(\Omega)}&\leq C_1 e^{CM^{1/2}} 
\int^{t_2}_{t_1+\frac{t_2-t_1}{4}}  \mathds{1}_E(t) \| e^{t\tilde{\Delta}} f\|_{L^2(\omega)}\, dt \nonumber \\& \qquad +e^{ \frac{-2a }{(t_2-t_1)}}\| e^{t_1\tilde{\Delta}} f\|_{L^2(\Omega)}.
\end{align}
It follows that
\begin{align}
e^{-\frac{a}{t_2-t_1}} \| e^{t_2\tilde{\Delta}} f\|_{L^2(\Omega)}&-e^{ \frac{-2a }{(t_2-t_1)}}\| e^{t_1\tilde{\Delta}} f\|_{L^2(\Omega)} \nonumber \\&\leq C_1 e^{C M^{1/2}} 
\int^{t_2}_{t_1+\frac{t_2-t_1}{4}} \mathds{1}_E(t) \| e^{t\tilde{\Delta}} f\|_{L^2(\omega)}\, dt.
\label{545-1}
\end{align}
 Recall that we have chosen $\gamma=2$, $t_1=k_{m+1}$ and $t_2=k_{m}$. Observe that in Lemma \ref{densi}
\begin{align}
\label{kkk-1}
\frac{1}{k_{m+1}-k_{m+2}}=\frac{2}{k_{m}-k_{m+1}}.
\end{align}
We have
\begin{align}
e^{-\frac{a}{k_{m}-k_{m+1}}} \| e^{k_m \tilde{\Delta}} f\|_{L^2(\Omega)}&-e^{ \frac{-a }{k_{m+1}-k_{m+2}}}\| e^{t_{k_{m+1}}\tilde{\Delta}} f\|_{L^2(\Omega)} \nonumber \\ & \leq C_1 e^{C M^{1/2}}
\int^{k_m}_{k_{m+1}} \mathds{1}_E(t)\| e^{t\tilde{\Delta}} f\|_{L^2(\omega)}\, dt.
\end{align}
Note that $ e^{-\frac{a}{k_{m}-k_{m+1}}}\to 0 $ as $m\to \infty$. Summing up the above telescopic series from $m=1$ to $\infty$ gives that
\begin{align}
e^{-\frac{a}{k_{1}-k_{2}}} \| e^{k_1 \tilde{\Delta}} f\|_{L^2(\Omega)}\leq  C_1 e^{CM^{1/2}}\int^{k_1}_{k} \mathds{1}_E(t)\| e^{t\tilde{\Delta}} f\|_{L^2(\omega)}\, dt.
\label{more-1}
\end{align}
Therefore, by the fact that $k_1 - k_2$ is a positive constant depending on $E$, we have
\begin{align}
\| e^{T\tilde{\Delta}} f\|_{L^2(\Omega)}\leq   e^{C(E) + C M^{1/2}}
\int_{E} \| e^{t\tilde{\Delta}} f\|_{L^2(\omega)}\, dt,
\end{align}
where we also used \eqref{reg-first}. Note that in the above inequality, $C$ depends on $\omega$ and $\Omega$, while $C(E)$ depends additionally on $E$.
Let $f=y_0$ in (\ref{nonlinear-one}). We arrive at
\begin{align}
\| y(\cdot, T)\|_{L^2(\Omega)}\leq  e^{C(E) + CM^{1/2}}
\bigg(\int_E \int_\omega y^2 \, dx dt \bigg)^\frac{1}{2}.
\label{mama-b}
\end{align}
Hence, we complete the proof of Case 1 in Theorem \ref{th3}.

 For the sign-changing potential $V(x)$ with $\|V\|_\infty=M$, we can reduce to the nonnegative potential in Case 1. Let $\hat{y}(x,t)= e^{-\|V_-\|_\infty t} y(x,t)$. Then $\hat{y}(x,t)$ satisfies 
\begin{align}
  \hat{y}_t- \partial_{xx}\hat{y}+(\|V_-\|_\infty+V(x)) \hat{y}=0.
\end{align}
Since $\|V_-\|_\infty+V(x)\geq 0$,  from (\ref{mama-b}), 
\begin{align}
\| \hat{y}(\cdot, T)\|_{L^2(\Omega)}\leq  e^{C(E) + CM^{1/2}}
\bigg(\int_E \int_\omega \hat{y}^2 \, dx dt \bigg)^\frac{1}{2}.
\end{align}
The definition of $\hat{y}$ implies that
\begin{align}
\| {y}(\cdot, T)\|_{L^2(\Omega)}\leq  e^{C(E) + T\| V_- \|_\infty+CM^{1/2}}
\bigg(\int_E \int_\omega {y}^2 \, dx dt \bigg)^\frac{1}{2}.
\label{end-1}
\end{align}
Hence, we complete the proof of Theorem \ref{th3}.
\end{proof}

\begin{proof}[Proof of Corollary \ref{cor1}] Following from the arguments in the proof of Theorem \ref{th3}, we check the term
$e^{-\frac{a}{k_{1}-k_{2}}}$ in (\ref{more-1}).
If $E=(0, T)$, we set $k_{m+1}= 2^{-m} T$ and hence $k_1 -k_2 = \frac12 T$. Also recall that $a$ in (\ref{more-1})  is an absolute constant  depending on  $\Omega$ and $\omega$. Thus, \eqref{obs-0T-V} follows. 
%For Case 2, The discussion in Case 1 and (\ref{end-1}) imply the conclusion \eqref{obs-0T-V}.
\end{proof}

\section*{Appendix}
%\appendix
\renewcommand{\theequation}{A.\arabic{equation}}
\setcounter{equation}{0}
\renewcommand{\thetheorem}{A.\arabic{theorem}}
\setcounter{theorem}{0}

In Appendix we provide two lemmas on the regularity estimates for heat equations with emphasis on how the constants depend on the potential $V$. Note that $\| V \|_Z=\|V\|_\infty + \|\nabla V\|_\infty+ \|\partial_tV\|_\infty $.

% We study the regularity estimates and show how the estimates depends on the norm of potential function $V(x,t)$  in the appendix.
% We consider the homogeneous equation
% \begin{equation}
%     \left\{
%     \begin{aligned}
%         y_t-\Delta y+V(x, t)y &= 0 \quad  &\text{in }& \ Q_T, \\
%         y &=0  \quad &\text{on }& \ \Sigma_T, \\
%         y(\cdot, 0) &= y_0    &\text{on }& \  \Omega.
%     \end{aligned}
%     \right.
%     \label{basic-in-2}
% \end{equation}

\begin{lemma}
Let $y_0\in C^{2+\alpha}(\overline{\Omega})$ with some $\alpha \in (0,1)$ and $V \in Z$. Let $y$ be the solution of (\ref{nonlinear-core}). %$V(x,t)\in L^\infty(0, T; W^{1, \infty})$ and $V_t\in L^\infty(0, T; L^{\infty})$. 
Then there exists some  constant $C=C(\Omega, \alpha)$ depending only on $\Omega$ and $\alpha$ such that
\begin{align}
 \|y\|_{C^{{2+\alpha}, {1+\frac{\alpha}{2}}}(\overline{Q}_T)} \leq C e^{ (3+T) \|V\|_{\infty}}(1+\|V\|_{Z}) \| y_0\|_{C^{2+\alpha}(\Omega)}.
\end{align}
\label{lemm-reg}
\end{lemma}
\begin{proof}
We apply a well-known bootstrap argument to get the higher order regularity with emphasis on how the constant depends on the potential $V$.
First of all, the $L^\infty$ bound can be shown simply by the maximal principle for the heat equation \cite[Theorem 2.10]{L96}, i.e.,
\begin{equation}\label{est.MP}
    \|y\|_{L^\infty(0, T; L^\infty)}\leq e^{T\| V\|_{\infty}} \| y_0\|_{L^\infty(\Omega)}.
\end{equation}

Next, we apply the classical $L^p$ estimate of heat equations to \eqref{nonlinear-core} to obtain the $W^{2, 1; p}$ estimate, see e.g. \cite{L96}. If $T<1$, we may extend the Lipschitz potential from $Q_T$ to $\Omega\times (0,1)$ such that the global $W^{2, 1; p}$ estimate can be applied in $\Omega\times (0,1)$. If $T>1$, we apply the global $W^{2, 1; p}$ estimate in $\Omega\times (0,1)$ to get
\begin{equation}\label{est.w21p-1}
\begin{aligned}
    \|y\|_{W^{2, 1, p}(\Omega\times (0,1))}  & \leq C (\|Vy\|_{L^\infty(\Omega\times (0,1))} +\| y_0\|_{C^{2}(\Omega)}) \\
    & \le C(1+ \| V\|_\infty ) e^{\| V\|_\infty} \| y_0\|_{C^{2}(\Omega)},
\end{aligned}
\end{equation}
where we have used \eqref{est.MP} in the second inequality.
For $\Omega\times (t,t+1)$ with $t\geq 1$, we apply the interior $W^{2,1;p}$ estimate (local in time) to obtain
\begin{equation}\label{est.w21p-2}
\begin{aligned}
    \|y\|_{W^{2, 1, p}(\Omega\times (t,t+1))}  & \leq C ( \|Vy\|_{L^\infty(\Omega\times (t-1,t+1))} + \| y\|_{L^\infty(\Omega\times (t-1,t+1))} ) \\
    & \leq C(1 + \| V\|_\infty ) e^{(t+1) \| V \|_\infty} \| y_0 \|_{L^\infty(\Omega)},
\end{aligned}
\end{equation}
where we have used \eqref{est.MP} again in the second inequality. 

% Combining \eqref{est.w21p-1} and \eqref{est.w21p-2}, we obtain
% \begin{equation}
%     \|y\|_{W^{2, 1, p}(Q_T)} \le C(1 + \| V\|_\infty ) e^{C (T+1) \| V \|_{\infty} } \| y_0 \|_{C^2(\Omega)}.
% \end{equation}

% For large time $T>1$, we may repeatedly apply the estimate in $\Omega \times (t,t+1)$. As a consequence,
% \begin{align}
% \begin{aligned}
%    \|y\|_{W^{2, 1, p}(Q_T)} & \leq e^{\jz{C(1+T)}} (\|Vy\|_{L^\infty(Q_T)} +\| y_0\|_{C^{2}(\Omega)}) \\
%    & \leq e^{C(1+T)} \big( \| V\|_{\infty} e^{T\| V\|_{\infty}} \| y_0\|_{L^\infty(\Omega)} + \| y_0\|_{C^{2}(\Omega)}) \big),
% \end{aligned}
% \end{align}
% for any $1<p<\infty$, where $C$ depends on $p$, $n$ and $\Omega$. Hence we obtain
% \begin{align*}
% \|y\|_{W^{2, 1, p}(Q_T)}\leq e^{C(\frac{1}{T}+T + T \|V\|_{L^\infty})}(1+ \|V\|_{L^\infty(Q_T)}) \| y_0\|_{C^{2+\alpha}}.
% \end{align*}

Now, we use the Sobolev embedding theorem to get
\begin{equation}\label{est.C1beta-1}
    \|  y\|_{C^{1+\beta, \frac{1+\beta}{2}}(\Omega \times(0,1))} \le C\|y\|_{W^{2, 1, p}(\Omega\times (0,1))},
\end{equation}
and for any $t\geq 1$,
\begin{equation}\label{est.C1beta-2}
    \|  y\|_{C^{1+\beta, \frac{1+\beta}{2}}(\Omega \times(t,t+1))} \le C\|y\|_{W^{2, 1, p}(\Omega\times (t,t+1))},
\end{equation}
where 
% we have
% \begin{align*}
% \|y\|_{C^{{1+\beta}, \frac{1+\beta}{2}}(\overline{Q}_T)}\leq e^{C(\frac{1}{T}+T)} \|y\|_{W^{2, 1; p}}
% \end{align*}
$p>n+2$ and $0<\beta=1-\frac{n+2}{p}<1$. 

Finally, we apply the $C^{2+\alpha,1+\frac{\alpha}{2}}$ estimate in $\Omega \times (0,1)$ and $\Omega\times (t,t+1/2)$ for each $t\geq 1$. For the former case, we have
\begin{equation}\label{est.C2alpha-1}
    \|y\|_{C^{2+\alpha,1+\frac{\alpha}{2}}(\Omega\times (0,1))}   \leq C (\|Vy\|_{C^{\alpha, \frac{\alpha}{2}}(\Omega\times (0,1))} +\| y_0\|_{C^{2+\alpha}(\Omega)}).
\end{equation}
Using \eqref{est.MP}, \eqref{est.w21p-1} and \eqref{est.C1beta-1}, we have
\begin{equation*}
\begin{aligned}
    \|Vy\|_{C^{\alpha, \frac{\alpha}{2}}(\Omega\times (0,1))} & \le \|V\|_{C^{\alpha, \frac{\alpha}{2}}(Q_T)} \| y\|_{L^\infty(\Omega\times (0,1))} + \| V \|_{L^\infty(\Omega\times (0,1))}\|y \|_{C^{\alpha, \frac{\alpha}{2}}(\Omega\times (0,1))} \\
    & \le \|V\|_{C^{\alpha, \frac{\alpha}{2}}(Q_T)} e^{\| V \|_\infty} \| y_0\|_{L^\infty(\Omega)}  + C\| V\|_\infty e^{2\| V\|_\infty} \| y_0 \|_{C^2(\Omega)} \\
    & \le C\| V\|_Z e^{3\|V\|_\infty} \| y_0\|_{C^2(\Omega)}.
\end{aligned}
\end{equation*}
Inserting this into \eqref{est.C2alpha-1}, we obtain
\begin{equation}
    \|y\|_{C^{2+\alpha,1+\frac{\alpha}{2}}(\Omega\times (0,1))} \le C(1+\| V\|_Z) e^{3\|V\|_\infty} \| y_0 \|_{C^{2+\alpha}(\Omega)}.
\end{equation}
For the latter case in $\Omega\times (t,t+1/2)$ with $t\geq 1$, it follows from a similar argument that
\begin{equation}
    \|y\|_{C^{2+\alpha,1+\frac{\alpha}{2}}(\Omega\times (t,t+1/2))} \le C(1+\| V\|_Z) e^{(t+3)\|V\|_\infty} \| y_0 \|_{L^\infty(\Omega)}.
\end{equation}
As a consequence, we glue these estimates to get the desired estimate \eqref{lemm-reg}.
\end{proof}

By Proposition \ref{pro1}, there exists a control $\tilde{h}\in L^2(  \omega_2\times (0, T))$ such that the associated solution $\tilde{y}$ satisfies
\begin{equation}
    \left\{
    \begin{aligned}
        \tilde{y}_t-\Delta\tilde{y}+V(x,t)\tilde{y} &=\tilde{h} \mathds{1}_{\omega_2} \quad  &\text{in }& \ \Omega\times (0, T) , \\
        \tilde{y} &=0  \quad &\text{on }& \ \partial\Omega\times (0, T), \\
        \tilde{y}(\cdot,0) &= y_0    &\text{on }& \  \Omega,
    \end{aligned}
    \right.
    \label{appen-reg}
\end{equation}
and $y(\cdot,T) = 0$.
% Furthermore, it holds that
% \begin{align}
% \| \tilde{h}\|_{L^2(  \omega_2\times (0, T) )}
% &\leq e^{C(1+\frac{1}{T}+T \|V\|_{\infty}+\| V\|^{\frac{1}{2}}_{L^\infty(0, T;W^{1,\infty})}+\|\partial_t V\|^{\frac{1}{3}}_{L^\infty})} \|y_0\|_{C^{2+\alpha}(\overline\Omega)}.
% \end{align}
%Next we will show the quantitative regularity estimates for $\tilde{y}$ away from a given characteristic set $\omega_2$.
%Together with the estimates (\ref{schau-1}), we have
%\begin{align}
%\|y\|_{L^{q_1}(0, T; W^{1, q_1})}\leq C e^{T+\frac{1}{T}} (\|V\|_{L^\infty(Q_T)}  \|y\|_{L^{q_0}(Q_T)}+\| y_0\|_{C^{2+\alpha}(\Omega)}).
%\label{reg-com}
%\end{align}
In the proof of Theorem \ref{th2}, we need the following lemma.
\begin{lemma}
Assume $\tilde{y}$ satisfies (\ref{appen-reg}) and $\omega_2\Subset \omega_3 \Subset \Omega$. Then there exists an open set $\OO \subset \Omega$ such that
$\partial \omega_3\subset \OO$ and
% \begin{align}
% \|\tilde{y}\|_{C^{2+\alpha, 1+\frac{\alpha}{2}}(\OO\times (0, T))}\leq e^{C(T+\frac{1}{T}+T \|V\|_{\infty}+\| V\|^{\frac{1}{2}}_{L^\infty(0, T;W^{1,\infty})}+\|\partial_t V\|^{\frac{1}{3}}_{L^\infty})} \|y_0\|_{C^{2+\alpha}(\Omega)}
% \end{align}
\begin{equation}\label{est.intC2alpha}
    \|\tilde{y}\|_{C^{2+\alpha, 1+\frac{\alpha}{2}}(\OO\times (0, T))} \le C(1+\| V\|_Z) e^{C(T+1)\|V\|_\infty} (\| y_0\|_{C^{2+\alpha}(\Omega)} + \| \tilde{h}\|_{L^2(Q_T)}).
\end{equation}
\label{last-reg-lem}
\end{lemma}
\begin{proof}
    First of all, by the energy estimate, we have
    \begin{equation}
        \| \tilde{y} \|_{L^2(Q_t)} \le Ce^{Ct\| V\|_\infty} (\| y_0\|_{L^2(\Omega)} + \| \tilde{h}\|_{L^2(Q_t)} ),
    \end{equation}
    for any $t\in (0,T]$.
    Let $x_0 \in \partial \omega_3$ and $r>0$ such that $B_{2r}(x_0) \subset \Omega \setminus \omega_2$. Due to the compactness of $\partial \omega_3$, the radius $r>0$ can be chosen universally independent of $x_0 \in \partial \omega_3$. Let $\widetilde{Q}_r(x_0,t) = B_r(x_0) \times (t-r,t)$. Observe that if $t\geq 2r$, we have
    \begin{equation}
        \tilde{y}_t-\Delta\tilde{y}+V(x,t)\tilde{y}=0, \qquad \text{in } \widetilde{Q}_{2r}(x_0,t).
    \end{equation}
    Thus we can get the $C^{2+\alpha,1+\frac{\alpha}{2}}$ estimate in $\widetilde{Q}_r(x_0,t)$ by the classical interior Schauder theory. To clarify how the constant depends on the potential $V$, we may  apply a bootstrap argument as in the proof of Lemma \ref{lemm-reg} to get
    \begin{equation}\label{est.localC2alpha-1}
    \begin{aligned}
        \| \tilde{y} \|_{C^{2+\alpha,1+\frac{\alpha}{2}}(\widetilde{Q}_{r}(x_0,t))} & \le  C(1+\| V\|_Z) e^{C(t+1)\|V\|_\infty} \| \tilde{y} \|_{L^2(\widetilde{Q}_{2r}(x_0,t))} \\
        & \le C(1+\| V\|_Z) e^{C(t+1)\|V\|_\infty} (\| y_0\|_{L^2(\Omega)} + \| \tilde{h}\|_{L^2(Q_t)}).
    \end{aligned}
    \end{equation}
    The constant $C$ possibly depends on the radius $r$, which is harmless since it has a uniform lower bound. 
    %The exponent $N>0$, depending only on the dimension, arises through the bootstrap argument.

    Similarly, near $t = 0$, we consider
    \begin{equation}
    \left\{
    \begin{aligned}
        \tilde{y}_t-\Delta\tilde{y}+V(x,t)\tilde{y} &=0 \quad  &\text{in }& \ B_{2r}(x_0) \times (0, 2r) , \\
        \tilde{y}(\cdot,0) &= y_0    &\text{on }& \  B_{2r}(x_0).
    \end{aligned}
    \right.
\end{equation}
Then the interior Schauder estimate with smooth initial data yields 
%\textcolor{blue}{(please show a little more details about the estimates in the first inequality below)}
\begin{equation}\label{est.localC2alpha-2}
\begin{aligned}
    & \|\tilde{y}\|_{C^{2+\alpha,1+\frac{\alpha}{2}}(B_r(x_0 )\times (0,2r))} \\
    & \le C(1+\| V\|_Z) e^{C \|V\|_\infty} ( \| y_0 \|_{C^{2+\alpha}(B_{2r}(x_0))}  + \| \tilde{y} \|_{L^2(B_{2r}(x_0) \times (0, 2r))} ) \\
    & \le C(1+\| V\|_Z) e^{C \|V\|_\infty} ( \| y_0 \|_{C^{2+\alpha}(\Omega)} + \| \tilde{h} \|_{L^2(Q_{2r})}).
\end{aligned}
\end{equation}
Gluing the estimate \eqref{est.localC2alpha-1} with all $t\in (2r,T)$ and $x_0 \in \partial \omega_3$ with \eqref{est.localC2alpha-2}, we get \eqref{est.intC2alpha} with $\OO = \cup_{x_0 \in \partial\omega_3} B_r(x_0)$.
\end{proof}

\small
\bibliography{Mybib}
\bibliographystyle{abbrv}
\end{document}